\documentclass[11pt]{article}

\usepackage{amssymb,amsmath,mathtools}%,mathrsfs}
 \usepackage{amsthm} %for theorem style
 \usepackage{mathrsfs}

\usepackage[OT2, T1]{fontenc} % UTF8
\usepackage[utf8]{inputenc}

\usepackage[italian, english]{babel}

\usepackage[a4paper,textwidth=180true mm,lines=43]{geometry}

\usepackage{xcolor}
\usepackage{bm}

\usepackage{braket}
\usepackage{enumitem}

\allowdisplaybreaks %breaks equations in different pages

\usepackage[mathlines]{lineno} %for line numbers
\usepackage
[draft=false,setpagesize=false,pdfstartview=FitH,
colorlinks=true,
linkcolor=blue, %black
citecolor=magenta,
pagebackref=true,
bookmarks=true]
{hyperref}

\usepackage{aliascnt}
\usepackage[nameinlink,capitalise]{cleveref}

\usepackage[hyperpageref]{backref} %Insert in the biliography the pages where the paper is cited.

\usepackage[noadjust]{cite} %to remove space in citations, e.g. in \begin{Theorem}[\cite{, and puts together citations, e.g. 37--39

\usepackage{footnotebackref} %to have a hyperlink to the point in the text where the footnote started.

\makeatletter
\renewenvironment{proof}[1][\proofname]{\par
 \pushQED{\qed}%
 \normalfont \topsep6\p@\@plus6\p@\relax
 \trivlist
 \item\relax
 {\bfseries
 #1\@addpunct{.}}\hspace\labelsep\ignorespaces
}{%
 \popQED\endtrivlist\@endpefalse
}
\makeatother

\theoremstyle{theorem}
\newtheorem{Theorem}{Theorem}[section]
\crefname{Theorem}{Theorem}{Theorems}

\newaliascnt{Lemma}{Theorem}
\newtheorem{Lemma}[Lemma]{Lemma}
\aliascntresetthe{Lemma}
\crefname{Lemma}{Lemma}{Lemmas}

\newaliascnt{Proposition}{Theorem}
\newtheorem{Proposition}[Proposition]{Proposition}
\aliascntresetthe{Proposition}
\crefname{Proposition}{Proposition}{Propositions}

\newaliascnt{Corollary}{Theorem}
\newtheorem{Corollary}[Corollary]{Corollary}
\aliascntresetthe{Corollary}
\crefname{Corollary}{Corollary}{Corollaries}

\newaliascnt{Remark}{Theorem}
\newtheorem{Remark}[Remark]{Remark}
\aliascntresetthe{Remark}
\crefname{Remark}{Remark}{Remarks}

\theoremstyle{definition}
\newaliascnt{Definition}{Theorem}
\newtheorem{Definition}[Definition]{Definition}
\aliascntresetthe{Definition}
\crefname{Definition}{Definition}{Definitions}

\newaliascnt{Example}{Theorem}

\aliascntresetthe{Example}
\crefname{Example}{Example}{Examples}

\crefname{Section}{Section}{Sections}
\crefname{Subsection}{Subsection}{Subsections}

\crefformat{Theorem}{Theorem~#2#1#3}
\crefformat{Lemma}{Lemma~#2#1#3}
\crefformat{Proposition}{Proposition~#2#1#3}
\crefformat{Corollary}{Corollary~#2#1#3}
\crefformat{Remark}{Remark~#2#1#3}
\crefformat{Definition}{Definition~#2#1#3}
\crefformat{Example}{Example~#2#1#3}

\crefformat{section}{Section~#2#1#3}
\crefmultiformat{section}{Sections~#2#1#3}{ and~#2#1#3}{, #2#1#3}{ and~#2#1#3}
\crefrangeformat{section}{Sections~#3#1#4--#5#2#6}

\crefformat{subsection}{Subsection~#2#1#3}
\crefmultiformat{subsection}{Subsections~#2#1#3}{ and~#2#1#3}{, #2#1#3}{ and~#2#1#3}
\crefrangeformat{subsection}{Subsections~#3#1#4--#5#2#6}

\numberwithin{equation}{section}

\newcommand{\R}{\mathbb{R}}
\newcommand{\N}{\mathbb{N}}

\def\abs#1{|{#1}|}

\def\norm#1{\|{#1}\|}

\def\wlimit{\rightharpoonup}

\def\PSPC{\mathrm{(PSPC)}}

\def\half{{\frac{1}{2}}}

\usepackage{scalerel}
\usepackage{tikz}
\usetikzlibrary{svg.path}
\definecolor{orcidlogocol}{HTML}{A6CE39}
\tikzset{
 orcidlogo/.pic={
 \fill[orcidlogocol] svg{M256,128c0,70.7-57.3,128-128,128C57.3,256,0,198.7,0,128C0,57.3,57.3,0,128,0C198.7,0,256,57.3,256,128z};
 \fill[white] svg{M86.3,186.2H70.9V79.1h15.4v48.4V186.2z}
 svg{M108.9,79.1h41.6c39.6,0,57,28.3,57,53.6c0,27.5-21.5,53.6-56.8,53.6h-41.8V79.1z M124.3,172.4h24.5c34.9,0,42.9-26.5,42.9-39.7c0-21.5-13.7-39.7-43.7-39.7h-23.7V172.4z}
 svg{M88.7,56.8c0,5.5-4.5,10.1-10.1,10.1c-5.6,0-10.1-4.6-10.1-10.1c0-5.6,4.5-10.1,10.1-10.1C84.2,46.7,88.7,51.3,88.7,56.8z};
 }
}
\newcommand\orcidicon[1]{\href{https://orcid.org/#1}{\mbox{\scalerel*{
\begin{tikzpicture}[yscale=-1,transform shape]
\pic{orcidlogo};
\end{tikzpicture}
}{|}}}}

\def\calC{{\cal C}} % just used for C(L) in the definition of \oGamma

\def\calI{{\cal I}}
\def\calM{{\cal M}}

\def\calS{{\cal S}}

\def\scrP{\mathscr{P}}

\def\epsilon{\varepsilon}

\DeclareMathOperator{\sgn}{sgn}

\def\RE{\R\times E}

\def\intRN{\int_{\R^N}}

\def\distE{{\rm dist}_E}

\def\ob{\overline b}
\def\ub{\underline b}

\def\oGamma{\overline \Gamma}
\def\uGamma{\underline\Gamma}

\newcommand{\RN}{\R^N}
\newcommand{\mc}[1]{\mathcal{#1}}
\def\supp{\mathop{\rm supp}}

\def\calK{{\cal K}}
\def\calG{{\cal G}}

\def\eps{\varepsilon}

\begin{document}

\title
{Normalized ground states for NLS equations \\ with mass critical nonlinearities}

\author{
 Silvia Cingolani%
%\footnote{%
\thanks{%
Corresponding author: \href{mailto:silvia.cingolani@uniba.it}{silvia.cingolani@uniba.it}.}$\,$
 	\orcidicon{0000-0002-3680-9106}
 \\ \small{Dipartimento di Matematica, Universit\`{a} degli Studi di Bari Aldo Moro}
 \\ \small{Via E. Orabona 4, 70125, Bari, Italy}
%\\ \small{\href{mailto:silvia.cingolani@uniba.it}{silvia.cingolani@uniba.it}}
 \\ \vspace{-0.5em}
 \\Marco Gallo
 	\orcidicon{0000-0002-3141-9598}
 \\ \small{Dipartimento di Matematica e Fisica, Universit\`{a} Cattolica del Sacro Cuore}
 \\ \small{Via della Garzetta 48, 25133 Brescia, Italy}
%\\ \small{\href{mailto:marco.gallo1@unicatt.it}{marco.gallo1@unicatt.it}}
 \\ \vspace{-0.5em}
 \\Norihisa Ikoma
 	\orcidicon{0000-0002-3547-5650}
 \\ \small{Department of Mathematics, Faculty of Science and Technology, Keio University}
 \\ \small{Yagami Campus, 3-14-1 Hiyoshi, Kohoku-ku, Yokohama, Kanagawa 223-8522, Japan}
%\\ \small{\href{mailto:ikoma@math.keio.ac.jp}{ikoma@math.keio.ac.jp}}
 \\ \vspace{-0.5em}
 \\Kazunaga Tanaka 
 	\orcidicon{0000-0002-1144-1536}
 \\ \small{Department of Mathematics, School of Science and Engineering, Waseda University}
 \\ \small{3-4-1 Ohkubo, Shijuku-ku, Tokyo 169-8555, Japan}
%\\ \small{\href{mailto:kazunaga@waseda.jp}{kazunaga@waseda.jp}}
 } 

\date{}

\maketitle

\abstract{
We study normalized solutions $(\mu,u)\in \mathbb{R} \times H^1(\mathbb{R}^N)$ to nonlinear Schrödinger equations 
 \[	-\Delta u + \mu u = g(u)\quad \hbox{in}\ \mathbb{R}^N, \qquad
 \frac{1}{2}\int_{\mathbb{R}^N} u^2 dx = m,
 \]
where $N\geq 2$ and the mass $m>0$ is given. Here $g$ has an $L^2$-critical growth, both at the origin and at infinity, that is $g(s)\sim |s|^{p-1}s$ as $s\sim 0$ and $s\sim\infty$, where $p=1+\frac{4}{N}$. 
We continue the analysis started in \cite{CGIT24}, where we found two (possibly distinct) minimax values $\underline{b} \leq 0 \leq \overline{b}$ of the Lagrangian functional.
In this paper we furnish explicit examples of $g$ satisfying $\underline{b}<0<\overline{b}$, $\underline{b}=0<\overline{b}$ and $\underline{b}<0=\overline{b}$; notice that $\underline{b}=0=\overline{b}$ in the power case $g(t)=|t|^{p-1}t$. Moreover, we deal with the existence and non-existence of a solution with minimal energy. 
Finally, we discuss the assumptions required on $g$ to obtain the existence of a positive solution for perturbations of $g$. 
%\blue{Finally, we discuss the assumptions required on $g$ by studying stability properties for existence. /// 
%Finally we discuss perturbations of $g$ and obtain the existence of a positive solution}
}

\medskip

\noindent
\textbf{MSC2020:} 
35A01, %Existence problems for PDEs: global existence, local existence, non-existence
35B20, %Perturbations in context of PDEs
35B33, %Critical exponents in context of PDEs
35B35, %Stability in context of PDEs
35B38, %Critical points of functionals in context of PDEs (e.g., energy functionals)
35J20, %Variational methods for second-order elliptic equations
35J91, %Semilinear elliptic equations with Laplacian, biLaplacian or poly-Laplacian
35Q40, %PDEs in connection with quantum mechanics
35Q55, %NLS equations (nonlinear Schr¨odinger equations)
47F10, %Elliptic operators and their generalizations
47J30, %Variational methods involving nonlinear operators
49J35. %Existence of solutions for minimax problems
%35A01, 35B20, 35B33, 35B35, 35B38, 35J20, 35J91, 35Q40, 35Q55, 47F10, 47J30, 49J35

\noindent
\textbf{Key words:} 
Nonlinear Schrödinger equations, 
nonlinear elliptic PDEs, 
normalized solutions, 
prescribed mass problem, 
$L^2$-critical exponent, 
least energy,
$L^2$-minimum,
Lagrangian approach,
stability and perturbation properties.

\bigskip

\bigskip

\newpage

\tableofcontents

%%%%%%%%%%%%%%%%%%%%%%%%%%%%%%%%%%%%%%%%%%%%%%%%%%%%%%%%%%%%%%%%

%%%%%%%%%%%%%%%%%%%%%%%%%%%%%%%%%
%%%%%%%%%%%%%%%%%%%%%%%%%%%%%%%%%
\section{Introduction}
%%%%%%%%%%%%%%%%%%%%%%%%%%%%%%%%%
%%%%%%%%%%%%%%%%%%%%%%%%%%%%%%%%%

Aim of this paper is to continue the analysis started in \cite{CGIT24} and study existence, non-existence and explicit examples to the following equation: 
 \begin{equation}\label{eq_main}
 \begin{dcases}
 -\Delta u + \mu u = g(u) & \hbox{in $\R^N$}, 
 \\
 \frac{1}{2} \intRN u^2 \, dx = m,&
 \end{dcases}
 \end{equation}
where $N\geq 2$, $g:\,\R\to\R$ and the mass $m>0$ are given, while $(\mu,u)\in\R\times H^1(\R^N)$ is unknown. 
The literature on the topic of normalized solutions is broad, and we confine to mention \cite{Soa20, BaVa13, BaSo17, JeLe22, WeWu22, JeLu22CV, HiTa19, DST23, JZZ24, MeSc24, BiMe21, Sch22, BDS25}.
Here we will focus on $g$ satisfying 
\[g(s)\sim \abs s^{p-1}s \quad \hbox{as $s\sim 0$ and $s\sim\infty$}, \qquad p\coloneqq 1+\frac{4}{N},\]
in the sense clarified below. Notice that we assume $\lim_{s \to 0} \frac{g(s)}{\abs{s}^{p-1}s} = \lim_{s \to \infty} \frac{g(s)}{\abs{s}^{p-1}s}=1
$, 
while the case in which these two limits are finite but different has been considered in \cite{JZZ24, Sch22}; we refer also to the recent \cite{ScSm24} where the same problem in dimension $N=1$ has been treated.
For a general discussion about this problem and the related literature we refer to \cite{CGIT24}. 
We start by stating our set of assumptions and recalling the main results obtained in \cite{CGIT24}. 

Let 
\[
h(s) \coloneqq g(s) - |s|^{p-1} s, \quad G(s)\coloneqq \int_0^s g(t) \,dt = \frac{1}{p+1} |s|^{p+1} + H(s), \quad H(s) \coloneqq \int_0^t h(t) \, dt
\]
and consider the auxiliary function
 \begin{equation}\label{eq_alpha}
 \rho(s) \coloneqq {\frac{H(s)}{s^2/2}}. %\qquad 
 \end{equation}
Moreover we set
\begin{equation*} %\label{1.16}
 F= F_p\coloneqq \big\{u \in L^{p+1} (\R^N) \mid \nabla u \in L^2(\R^{N}), \
 u(x) = u(\abs{x})\big\}
 \end{equation*}
 endowed with $\norm u_{F} \coloneqq \norm{\nabla u}_2 + \norm u_{p+1}$.
 Finally let $\omega_1$ be the unique, positive (radially symmetric) solution of 
 \begin{equation*} 
 -\Delta \omega+ \omega= \abs{\omega}^{p-1} \omega \quad \text{in}\ \R^N,
 \end{equation*}
 and write
 \begin{equation}\label{eq_m1}
 m_1 \coloneqq \frac{1}{2} \intRN \omega_1^2 \, dx. 
 \end{equation}

Let us consider the following assumptions.

\begin{enumerate}[label={\rm (g\arabic*)}]
 \setcounter{enumi}{-1}
  \item \label{(g0)} $g\in C(
 [0,+\infty)
 %\R
 )$.
 \end{enumerate}

\begin{enumerate}[label={\rm (g1$\alpha$)}] %[label={\rm (g1*)}]
\item \label{(g1*)} There exists $\alpha\in\R$ such 
that
 \begin{equation*}
 \lim_{s\to 
 0 %^+
 }{\frac{h(s)}{\abs s^{p-1}s}} =0, \quad
 \lim_{s\to 
+\infty
% \pm\infty
 }{\frac{h(s)}{s}} =\alpha.
 \end{equation*}
\end{enumerate}

\begin{enumerate}[label={\rm (g\arabic*)}]
 \setcounter{enumi}{1}
 
 \item \label{(g2)}
 If $u_0\in F$ satisfies $u_0\geq 0$ and
 \begin{equation} \label{eq_zero_mass} %\label{1.15}
 -\Delta u = g(u) \quad \text{in}\ \R^N, 
 \end{equation}
 then $u_0 \equiv 0$. 
\end{enumerate}

\begin{Remark}\label{Rem-on-g2}
We observe the following facts.
	\begin{enumerate}[label={\rm (\roman*)}]
	\item Since we are mainly interested in positive solutions, we require conditions in \ref{(g0)}, \ref{(g1*)} and \ref{(g2)} to hold in $[0,+\infty)$; indeed, it is sufficient to work with the odd extension of $g$ (see \ref{(g0')} below). On the other hand, to state non-existence theorems also of negative and sign-changing solutions, we will require some conditions also on $(-\infty, 0)$ (see \ref{(rho+1)} and \ref{(rho-1)} {\rm (ii)}).
    \item 
		Condition \ref{(g1*)} implies $ \lim_{s \to \infty} \rho(s)=\alpha$ where $\rho$ is defined in \eqref{eq_alpha}. 
        \item 
		The definition of the space $F$ depends on $p$, % in \ref{(g1*)} 
        and 
		condition \ref{(g2)} focuses on solutions belonging to $F$. 
		In this sense, the exponent $p$ also plays a role in \ref{(g2)}. 
		
		\item 
		 If $u \in F$ is a nontrivial non-negative solution of \eqref{eq_zero_mass} 
		under \ref{(g0)} and \ref{(g1*)}, 
		then elliptic regularity and the strong maximum principle imply that 
		$u$ is a positive classical solution of \eqref{eq_zero_mass}. 
	\end{enumerate}
\end{Remark}

One of the main results in \cite{CGIT24} is the following existence statement %result 
(see also \cref{thm_old_ex_precise} for a refined version).

\begin{Theorem}[{\cite[Theorem 1.1]{CGIT24}}] \label{thm_old_exist}
	Let $N\geq 2$, $m=m_1$ be given in \eqref{eq_m1}, and assume \ref{(g0)}, \ref{(g1*)} with $\alpha = 0$ and \ref{(g2)}.
	Then \eqref{eq_main} has a positive radially symmetric solution
	$(\mu,u)\in (0,\infty)\times H^1(\R^N)$.
\end{Theorem}

\begin{Remark}\label{Rem-val-g2}
	As pointed out in \cref{Rem-on-g2}, condition \ref{(g2)} depends on the exponent $p$ through $F$, 
	and the following hold (see \cite[Proposition 1.2]{CGIT24}):
	\begin{enumerate}[label={\rm (\roman*)}]
		\item 
		when $N=2,3,4$, \ref{(g0)} and \ref{(g1*)} with $\alpha = 0$ imply \ref{(g2)};
		\item 
		when $N \geq 3$, \ref{(g2)} is verified provided $g$ satisfies \ref{(g0)}, \ref{(g1*)} and \ref{(g3)}, where 
	\end{enumerate}
	\begin{enumerate}[label={\rm (g\arabic*)}]
		\setcounter{enumi}{2}
		\item \label{(g3)} 
			$G(s) \geq {\frac{N-2}{2N}}g(s)s$ for all $s> 0$.
	\end{enumerate}
	We remark that condition \ref{(g3)} also appears in \cite{BDS25,BiMe21,JeLu22CV,JZZ24,MeSc24}, while conditions similar to \ref{(g2)} appear in \cite{JZZ24} and \cite{ApMo25} (in this last case, the focus is on the exterior domain).
\end{Remark}

We next recall %discuss 
the non-existence of solutions to \eqref{eq_main} obtained in \cite{CGIT24}. 
For this purpose, we introduce the following conditions: 

\begin{enumerate}[label={\rm ($\rho_+$\arabic*)}]
	\setcounter{enumii}{1}
	\item \label{(rho+1)} 
		Condition \ref{(g1*)} holds with $\alpha>0$, $g \in C(\R)$, 
	$\rho(s) \leq \alpha$ for all $s \in \R \setminus \{0\}$ and 
	there exist $(s_n)_{n \in \N}$ and $(t_n)_{n \in \N}$ such that 
	$s_n < 0 < t_n$, $s_n,t_n \to 0$, $\rho (s_n) < \alpha$ and $\rho(t_n) < \alpha$; 
\end{enumerate}
\begin{enumerate}[label={\rm ($\rho_-$\arabic*)}]
\item \label{(rho-1)}
Condition \ref{(g1*)} holds with $\alpha<0$, and 
there exists $s_0>0$ such that
\begin{itemize}
\item[(i)] $s \mapsto \rho(s)$ is non-increasing in $(0,+\infty)$, and strictly decreasing in $(0,s_0]$;
\item[(ii)] $g \in C(\R)$, $s \mapsto \rho(s)$ is non-decreasing %increasing 
in $(-\infty,0)$, and strictly increasing %decreasing 
in $[-s_0,0)$.
\end{itemize}
\end{enumerate}
Notice that condition \ref{(rho+1)} is motivated by similar ones which can be found in \cite{BiMe21,BMS25,ScSm24}. 

\begin{Theorem}[{\cite[Theorem 1.7 and Remark 1.8 (iii)]{CGIT24}}] 
\label{thm_old_nonexist} %\label{Theorem:1.7}
Let $N\geq 2$, $m=m_1$ and assume \ref{(g0)}. 
If \ref{(rho-1)} {\rm (i)} holds, then no positive solution of \eqref{eq_main} exists. 
If \ref{(rho-1)} {\rm (i)} and {\rm (ii)} hold, then no solution of \eqref{eq_main} exists.
\end{Theorem}

We refer again to \cite{CGIT24} for several comments on the assumptions, on the choice $m=m_1$ and on the statements.
As highlighted in \cref{Rem-on-g2}, since we deal with positive solutions, it is not restrictive to assume that $g$ is odd. Namely, in the following discussion, we strengthen \ref{(g0)} by: 
\begin{enumerate}[label={\rm (g\arabic*')}]
	\setcounter{enumi}{-1}
	\item \label{(g0')} $g\in C(\R)$ and $g(-s)=-g(s)$ for all $s\in\R$.
\end{enumerate}
The existence result in \cref{thm_old_exist} %\cref{thm_old_nonexist} %above existence 
has been proved through a Lagrangian approach, introduced by Hirata and Tanaka \cite{HiTa19}.
For $m=m_1$ we consider the following Lagrangian functional:
\begin{equation}\label{eq_def-I}
	I(\lambda,u) \coloneqq \intRN \frac{1}{2} \abs{\nabla u}^2 - G(u) \, dx 
	+ e^\lambda \left( \frac{1}{2} \norm{u}_{2}^2 - m_1 \right) : \R \times H^1_r(\R^N) \to \R. 
\end{equation}
Here $H^1_r(\R^N)$ stands for the subspace of radially symmetric functions (see \eqref{eq_def_radiall_sub}). 
Notice that we are performing the identification $\mu \equiv e^{\lambda}$ to find a positive Lagrange multiplier in \eqref{eq_main}. 
We then consider the following minimax values (we give here just a rough idea of their definitions)
 \begin{equation*}
 \ub \coloneqq \inf_{\gamma \in \uGamma} \max_{0 \leq t \leq 1} I(\gamma(t)), \qquad 
 \ob \coloneqq \inf_{\gamma \in \oGamma} \sup_{ (t,\lambda) \in [0,1] \times \R } 
 I(\gamma(t, \lambda)),
 \end{equation*}
where 
 \begin{align*}
 &\uGamma \coloneqq \Big\{ \gamma \in C\big([0,1] , \R \times H^1_r(\R^N)\big) \;\Big|\; 
 \gamma(0)\in \R\times\{ 0\},\ I(\gamma(0)) \ll -1,\ 
 \gamma(1) \in \left(\mathrm{id} \times \zeta_0\right) (\R) \Big\}, 
 \\
 &\oGamma \coloneqq 
 \Big\{\gamma \in C\big( [0,1] \times \R , \R \times H^1_r(\R^N) \big) \;\Big|\; 
 \gamma(t,\lambda) = \gamma_0(t,\lambda) \ 
 \text{if $t\sim 0$ or $t\sim 1$ or $\abs{\lambda}\gg 1$} \Big\}. 
 \end{align*}
Here $\zeta_0 \in C^2(\R, H^1_r(\RN) )$ and $\gamma_0(t,\lambda)=(\lambda,t\zeta_0(\lambda))$ satisfy
\begin{itemize}
\item[(1)] $\gamma_0(0,\lambda)=(\lambda,0)$, 
$I(\gamma_0(0,\lambda))=I(\lambda,0)\leq 0$ 
for all $\lambda\in\R$;
\item[(2)] $\gamma_0(1,\lambda)=(\lambda,\zeta_0(\lambda))\in 
\R \times \{u \mid \half \norm{u}_2^2 > m_1 \}$ 
and $I(\gamma_0(1,\lambda))\ll -1%0
$ for all $\lambda\in\R$;
\item[(3)] $\max_{t\in [0,1]} I(\gamma_0(t,\lambda))\to 0$ 
as $\lambda\to\pm\infty$.
\end{itemize}
Notice that $\uGamma$ is a class of paths joining two points 
$\gamma(0)$ and $\gamma(1)$, 
which are inside and
outside the cylinder $\R\times\calS_0$, 
 \begin{equation} \label{eq_L2sphere}
 \calS_0 \coloneqq \left\{ u \in H^1_r(\R^N)\;\middle|\; 
 \half\intRN 
 u^2\,dx = m_1 \right\},
 \end{equation}
and satisfy $I(\gamma(0))$, $I(\gamma(1))\ll -1%0
$. 
On the other hand, $\oGamma$ %\red{M: it should be $\oGamma$} 
is a class of
2-dimensional paths $\gamma$ such that $\gamma([0,1]\times\R)$ links with $\{\lambda\}\times \calS_0$ for each $\lambda\in\R$, and at whose boundary (including $\lambda \approx \pm\infty$) 
$I$ is non-positive.
See \cref{Section_prelim} for precise definitions.

In \cite{CGIT24} 
the inequality $\ub \leq 0 \leq \ob$ is proved, 
and hence there are three cases to consider: 
 \[ \hbox{(1)\ $\ub < 0$}, \qquad \hbox{(2)\ $0 < \ob$}, \qquad 
 \hbox{(3)\ $\ub = 0 = \ob$}. 
 \]

For case (1) (resp. case (2)), we prove that $\ub$ (resp. $\ob$) is a critical 
value of $I$, see \cref{thm_old_ex_precise} below or \cite[Sections 7 and 8]{CGIT24}; notice that, if $\ub<0<\ob$, then actually existence of two positive 
solutions happens. 
For case (3) -- where a lack of compactness occurs -- we show the existence of uncountably many solutions (see \cite[Subsection 4.3]{CGIT24}); 
notice that (3) takes place if $g(s)=\abs s^{p-1}s$, in which case all these solutions are related by scaling.
We conjecture that this is the only possible case of $g$ for which (3) holds, but this still remains an open problem.

In this paper we address several goals. The first is to find some explicit examples in such a way (1) and (2) occur, possibly at the same time.

\begin{Theorem}%[\cref{Section_examples}]
\label{thm_examples}
There exist examples of $g$, satisfying \ref{(g0')}, \ref{(g1*)} with $\alpha=0$ and \ref{(g3)}, 
for each of the following cases:
\[
{\rm (i)} \; \; \ub<0<\ob \qquad {\rm (ii)} \; \; \ub=0<\ob \qquad {\rm (iii)} \; \; \ub<0=\ob \qquad {\rm (iv)} \; \; \ub=0=\ob.
\]
In particular, if {\rm (i)} occurs, then \eqref{eq_main} with $m=m_1$ admits two positive solutions in $(0,\infty)\times H^1_r(\R^N)$. 
\end{Theorem}

The second goal of this paper is to address a fundamental question on normalized solutions, that is the existence of a solution of minimal energy. 
Namely, setting the energy functional 
\begin{equation}\label{eq_def_energyf}
	\calI(u) \coloneqq \intRN \frac{1}{2} \abs{\nabla u}^2 - G(u) \, dx : H^1_r(\RN) \to \R, 
\end{equation}
we study the minimization %minimizing 
problem over $\calS_0$ (see \eqref{eq_L2sphere}), that is
 \begin{equation} \label{eq_min_L2} %\label{1.20}
 d \coloneqq \inf \big\{ \calI(u) \mid u \in \calS_0 \big\}.
 \end{equation}

\begin{Theorem}%[Existence, \cref{Section_L2min}]
\label{thm_exists_min} %\label{Theorem:1.5}
Assume \ref{(g0')} %\orange{M: we need \ref{(g0')}, right?}
and \ref{(g1*)}.
Then
\begin{enumerate}[label={\rm (\roman*)}]
	\item 	$d= \ub$ holds;
	
	\item 	suppose $\alpha = 0$ in \ref{(g1*)} % with $\alpha = 0$ 
	and $d<0$; 
	then $d$ is attained and a minimizer of \eqref{eq_min_L2} exists, which is positive and radially symmetric.
\end{enumerate}
\end{Theorem}

As a consequence of %corollary to
 \cref{thm_exists_min} (ii), we obtain the following result. 

\begin{Corollary}\label{cor_exists_min}
Let $g$ satisfy \ref{(g0')} and \ref{(g1*)} with $\alpha=0$. Assume also that 
\[
G(s) \geq \frac{1}{p+1}|s|^p \quad \text{for each $s \in [0,\infty)$}, %\text{for all}\ s\in\R
\]
where the inequality is strict in at least one point. 
Then a minimizer of \eqref{eq_min_L2} exists, which is positive and radially symmetric.
\end{Corollary}

\begin{Remark}\label{rem_odd_ext_minim}
We observe that, in \cref{cor_exists_min}, one may relax \ref{(g0')} to \ref{(g0)} by working with the odd extension of $g$, obtaining the following fact: let $g$ satisfy \ref{(g0)} and \ref{(g1*)} with $\alpha=0$, and $G(s) \geq \frac{1}{p+1}|s|^p$ for each $s \geq 0$, strict in at least one point; then there exists a positive, radially symmetric solution to \eqref{eq_main} which also minimizes
\[\inf \big\{ \calI(u) \mid u \in \calS_0, \; u \geq 0 \big\}.\]
Such a solution is found as minimizer of \eqref{eq_min_L2} where $G$ is substituted by its even extension.
\end{Remark}

Contrary to \cref{thm_exists_min}, we show that the negativity of $d$ is not a sufficient condition. Namely, the $L^2$-minimum could not be attained 
if \ref{(g1*)} holds with $\alpha > 0$; 
notice that this, together with \cref{thm_old_nonexist}, partially completes the analysis of the case $\alpha \neq 0$.
 
 \begin{Theorem}%[Nonexistence, \cref{Section_nonexist_min}] 
\label{thm_nonexis_min} %\label{Theorem:1.8}
	Let $N\geq 2$ and assume \ref{(g0)} and \ref{(rho+1)}. 
		Then $d = -\alpha m_1 <0$, where $d$ is defined in \eqref{eq_min_L2}. 
		However, $d$ is not attained.
\end{Theorem}

\begin{Remark}\label{rem_nonexistenc_result}
We observe the following facts. 
%\blue{M: I have made this remark a list.}
\begin{enumerate}[label={\rm (\roman*)}]
\item %Notice that, 
Differently from \cref{thm_exists_min}, in \cref{thm_nonexis_min} we do not require $g$ to be odd, nor we need
the second limit in \ref{(g1*)} to be satisfied also at $-\infty$. 
\item If condition \ref{(rho+1)} is required only for $s >0$, then in place of \cref{thm_nonexis_min} one may obtain that there exists no minimizer for $\inf \big\{ \calI(u) \mid u \in \calS_0, \; u \geq 0 \big\}$.
\item %In addition, 
In \cref{thm_nonexis_min}, 
we do not exclude the possibility that \eqref{eq_main} with $m=m_1$ has a solution that is not a minimizer of $\calI$ on $\mc{S}_0$. 
Such an existence remains an open problem.
\end{enumerate}
\end{Remark}

Before presenting the last main result of this paper, let us recall by \cite{CGIT24} that
\begin{equation}\label{eq_three_b}
\ub \leq b(\lambda) \leq \ob \quad \hbox{for each $\lambda \in \R$},
\end{equation}
where $b(\lambda)$ is the ground state level of the fixed frequency problem $u \mapsto I(\lambda, u)$ (see \eqref{eq_def_blambda}).

\begin{Remark}\label{rem_legendre_transf}
By \cite[Corollary 7.5]{CGIT24} we have %recall 
that 
\begin{equation}\label{eq_lagr_first}
\ub = \inf_{\lambda\in\R} b(\lambda).
\end{equation}
In light of \cref{thm_exists_min}, 
and by making explicit the dependence of $d$ on $m_1$, we can rewrite \eqref{eq_lagr_first} as (recall $\mu=e^{\lambda}$)
\[d(m_1) = \inf_{\mu \in (0,\infty)} \big( a(\mu) - \mu m_1\big) \]
where $a(\mu)$ is the ground state level of the fixed frequency action $u \mapsto \intRN \frac{1}{2} \abs{\nabla u}^2 - G(u) \, dx + \frac{\mu}{2} \norm{u}_{2}^2 $ (see \eqref{eq_def_amu}).
Such a \emph{Legendre transform} type relation was proved in Dovetta, Serra and Tilli \cite{DST23} for the power type nonlinearity including 
$L^2$-subcritical, $L^2$-critical and $L^2$-supercritical cases. 
Though it is not explicitly written, the same formula is also obtained in \cite[Proposition 5.7]{HiTa19}. 
See also \cite%[(1.2) and Proposition 1.6]
{GaSc25}.
\end{Remark}

In \cref{rem_legendre_transf} we saw the role of the infimum of $b(\lambda)$; in what follows, we will see that also the supremum 
\begin{equation}\label{eq_def_btilde}
\tilde{b}\coloneqq \sup_{\lambda \in\R} b(\lambda)
\end{equation}
plays a role.
The third goal is indeed to study the case $\ob > 0$ where \ref{(g2)} is of key importance in finding %plays an important role to find 
a positive solution 
(see \cref{prop_compact_Kb} and \cref{thm_old_ex_precise} below). Here we remark that according to \cite{CGIT24} %or 
(see also \cref{Section_prelim}), 
the value $\ob > 0$ can be defined under \ref{(g0)} and \ref{(g1*)}. However, \ref{(g2)} is not simple to verify in general for $N \geq 5$. 
As described in \cref{Rem-val-g2}, a simple condition to check \ref{(g2)} is \ref{(g3)}. 
In \cref{Thm:Exw/og2} below, we show that the criticality of the level %value 
$\ob > 0$ %being critical 
is stable under some perturbations of $g$. 
Via this result, which has an interest on its own, we may find an example of $g$ for which \eqref{eq_main} admits a solution, however \ref{(g3)} does not hold. 

To state \cref{Thm:Exw/og2}, we prepare some notations. To make the dependency of $\ob$, $\tilde{b}$ and $b(\lambda)$ on $G$ clearer, 
the notation $\ob_G$, $\tilde{b}_G$ and $b_G(\lambda)$ is used here. For a class of perturbations, we consider the following space %class 
of functions $X$ defined by 
\begin{equation}\label{def:X}
	X \coloneqq \left\{ \Xi \in C^1(\R) \;\middle|\; \Xi(0) =0, \ \text{$\Xi$ is even}, \ \lim_{s \to 0} \frac{ \Xi'(s) }{|s|^{p-1}s} = 0 = \lim_{|s| \to \infty} \frac{\Xi'(s)}{|s|} \right\}	
\end{equation}
and a norm on $X$ defined by 
\begin{equation}\label{def:normX}
	\| \Xi \|_{ X } \coloneq \sup_{ 0 < |s| \leq 1 } \frac{ |\Xi'(s) | }{|s|^p} + \sup_{|s| \geq 1} \frac{|\Xi'(s)|}{|s|}. 
\end{equation}
In the following \cref{Thm:Exw/og2} we will consider a fixed function $g_1$ with some assumptions, 
and a perturbation $\xi=\Xi'$, $\Xi \in X$, so that the desired example will be given by $g\coloneqq g_1+\xi$.
Since the continuous dependence of 
	$\Xi \in X \mapsto \ob_{G_1+\Xi}$
-- useful to ensure that 
$\ob_{G_1}>0$ implies $\ob_{G_1+\Xi}>0$
for $\Xi$ small -- is not totally clear, we will work under the stronger condition 
$\sup_{\lambda \in \R} b_{G_1}(\lambda)>0$
(see indeed \eqref{eq_three_b}).

\begin{Theorem}\label{Thm:Exw/og2}
	Assume that $g_1$ satisfies \ref{(g0')}, \ref{(g1*)} with $\alpha = 0$, \ref{(g2)} and 
\begin{equation}\label{eq_extra_blambda}
\tilde{b}_{G_{1}}
= \sup_{\lambda \in\R} b_{G_{1}}%G
(\lambda)>0.
\end{equation}
	Then %, for any $\Xi_1 \in X$, 
	there exists $\rho_1 = \rho_1(g_1%, \Xi_1
	) > 0$ %and $\Xi_1 \in X$ \orange{M: dependence of $\Xi_1$?}, $\Xi_1 \neq 0$,
	 such that,
	for any $\Xi \in X$ with $\| \Xi \|_X < \rho_1$, % and $|\Xi| \leq \Xi_1$, 
	problem \eqref{eq_main} with $g=g_1 + \Xi'$ admits a positive solution corresponding to $\ob_{G_1+\Xi} > 0 $. 
\end{Theorem}

We highlight that the existence perturbation result in \cref{Thm:Exw/og2} is a consequence of a non-existence perturbation result related to the zero mass problem \eqref{eq_zero_mass}, which holds under more general growth conditions on $g$, see \cref{section_stab_nonext_zero}.
Moreover, the proof of \cref{Thm:Exw/og2} exploits also some continuity property with respect to $g$, 
in particular the lower semicontinuity of 
$\tilde{b}_{G_1}$
(see \cref{Lem:conti}). % \cref{lem_stabil_positiv}).
We highlight also that, as abovementioned, the minimax values $\ob_{G_1+\Xi}$ may a priori depend on some fixed quantities which arise in the construction of the paths $\oGamma_{G_1+\Xi}$; a keypoint in \cref{section_exist_perturb} is that such construction can be made uniform with respect to $g_1$ and $\rho_1$.

As a particular case of \cref{Thm:Exw/og2}, we obtain the following statement, where the abstract condition \eqref{eq_extra_blambda} is indeed verified (see \cref{prop_b_neg_post}).
\begin{Corollary}\label{Cor:Exw/og2}
	Assume that $g_1$ satisfies \ref{(g0')}, \ref{(g1*)} with $\alpha = 0$, \ref{(g2)} (e.g. \ref{(g3)} holds) and 
	\begin{equation}\label{eq_condg_1}
		G_1(s) \leq \frac{1}{p+1} \abs{s}^{p+1} \quad \text{for each $s \in [0,\infty)$}, %\quad 
	\end{equation}
	strict in at least one point. 
	Then the conclusion of \cref{Thm:Exw/og2} holds.
\end{Corollary}

\begin{Remark}\label{rem_odd_extension}
Similarly to \cref{rem_odd_ext_minim}, we highlight that, in %\cref{Thm:Exw/og2} and 
\cref{Cor:Exw/og2}, one may relax \ref{(g0')} to \ref{(g0)} and consider the odd extension $\widetilde{g}_1$ of $g_1$ in order to get a positive solution for $g_1+\Xi'$ at the positive level $\ob_{\widetilde{G}_1+\widetilde{\Xi}}>0$, for any $\Xi$ %\in X$ 
satisfying the conditions in \eqref{def:X} (here $\Xi$ even is not needed, and $\widetilde{\Xi}$ is indeed the even extension of $\Xi$)
with $\| \Xi \|_X < \rho_1=\rho_1(\tilde{g}_1)$.
\end{Remark}

As announced, consequently to \cref{Thm:Exw/og2}, we get the following existence result which is not covered by \cref{Rem-val-g2} (ii) (see \cite[Proposition 1.2]{CGIT24}). 

\begin{Corollary}\label{Cor:ex-g}
	Let $N\geq 2$. %$N \geq 3$. 
	Then there exists a function $g$ with the following properties:
	\begin{itemize}
		\item 
		$g$ satisfies \ref{(g0')}, \ref{(g1*)} with $\alpha = 0$, but does not fulfill \ref{(g3)};
		\item 
		there exists a positive solution 
		$u$ of \eqref{eq_main} with $\mu = e^\lambda$ and $I(\lambda,u) = \ob_G>0$. 
	\end{itemize}
\end{Corollary}

\begin{Remark}
We notice, by \eqref{eq_three_b}, that $\ob_G\geq \tilde{b}_G$, thus \eqref{eq_extra_blambda} is indeed generally stronger than $\ob_G>0$.
On the other hand, recalled that, by \cite{CGIT24}, $%\sup_{\lambda \in\R} b_G(\lambda) 
\tilde{b}_G\geq \lim_{\lambda \to \pm\infty} b_G(\lambda)=0$, we see %highlight 
that condition \eqref{eq_extra_blambda} is optimal from the point of view of our discussion. As a matter of fact, assume $\tilde{b}_G %\sup_{\lambda \in\R} b_G(\lambda)
=0$: if $\ub_G<0$, then we have a solution without assuming \ref{(g2)} (see \cref{thm_old_ex_precise}); if $\ub_G=0$, then by \eqref{eq_three_b} we have $b_G(\lambda)=0$, and arguing as in \cite{CGIT24} we obtain again the existence of a solution without \ref{(g2)}. Thus the only case in which \ref{(g2)} (and thus \ref{(g3)}) plays a role is when \eqref{eq_extra_blambda} holds.
\end{Remark}

The paper is organized as follows. %\orange{TO BE WRITTEN}
In \cref{Section_prelim} we recall some constructions and some results contained in \cite{CGIT24}, together with some other classical results on Schrödinger equations. In \cref{Section_examples} we provide explicit examples of $g$ satisfying the different cases stated in \cref{thm_examples}. \cref{Section_L2min} is devoted to the study of the $L^2$-minimization problem, in particular in \cref{subsec_theorem_exist_min} we show the existence of a minimizer (\cref{thm_exists_min}), while in \cref{Section_nonexist_min} we show the non-existence (\cref{thm_nonexis_min}).
Finally, in \cref{sec_stability} we discuss %about 
some stability properties of the problem: in particular, in \cref{section_stab_nonext_zero} we focus on non-existence for the zero mass problem \eqref{eq_zero_mass}, while in \cref{section_exist_perturb} we consider existence for the main problem \eqref{eq_main} and prove \cref{Thm:Exw/og2} and \cref{Cor:ex-g}.

Let us highlight that, apart from \cref{Section_nonexist_min} %and \cref{section_stab_nonext_zero} 
where \ref{(g0)} is sufficient, in the rest of the paper we assume \ref{(g0')}.

%%%%%%%%%%%%%%%%%%%%%%%%%%%%%%%%%%%%%%%%%%%%%%%%%%%%%%%%%%%%%%%%%%
\section{Preliminaries}
\label{Section_prelim}

In this section we recall several results, some of them obtained in \cite{CGIT24}, to which we refer for details. 

Let us consider the subspace of radially symmetric functions in $H^1(\R^N)$
\begin{equation} \label{eq_def_radiall_sub}
	E\coloneqq H_r^1(\R^N)= \Big\{ u\in H^1(\R^N) \mid u(x)=u(\abs x) \big\}.
\end{equation}
For the $L^2$-sphere $\calS_0$ in \eqref{eq_L2sphere}, an $L^2$-cylinder $M_0$ in $\R\times E$ %$E \times \R$ 
is defined by 
 \begin{equation*} %\label{3.6}
 M_0 \coloneqq \R \times \calS_0. 
 \end{equation*}
 
 \medskip
 
\noindent
\textbf{I. Functionals.}

For each $\mu>0$, we define the action functional $\Psi_\mu:\, E\to\R$ by
\begin{equation*} %\label{3.5}
	\Psi_\mu(u) \coloneqq \frac{1}{2} \norm{\nabla u}_2^2 + \frac{\mu}{2} \norm{u}_2^2 
	- \intRN G(u) \, dx.
\end{equation*}
Notice that any critical point of $\Psi_\mu$ is a solution to 
\begin{equation}\label{eq:Psi_mu}
	-\Delta u + \mu u = g(u) \quad \text{in} \ \R^N. 
\end{equation}
By the identification 
\[\mu \equiv e^{\lambda}>0,\]
we recall the Lagrangian functional $I: \R \times E \to \R$ in \eqref{eq_def-I} given by % is rewritten as 
\begin{equation*}
	I(\lambda,u) = \intRN \frac{1}{2} \abs{\nabla u}^2 - G(u) \, dx 
	+ \mu \left( \frac{1}{2} \norm{u}_{2}^2 - m_1 \right) 
	= \Psi_\mu(u) - \mu m_1.
\end{equation*}
We highlight that such an identification among $\mu$ and $e^{\lambda}$ will be carried out throughout the paper. 
We denote its set of critical points at a level $b \in \R$ by 
\begin{equation} \label{eq_def_criticalset}%\label{7.1}
	K_b \coloneqq \big\{(\lambda,u)\in\RE \mid I(\lambda,u)=b, \ 
		D_{\lambda, u}I(\lambda, u)=0
	\big\}.
\end{equation}
Recalling the functional $\calI$ in \eqref{eq_def_energyf}, we clearly have 
\[
\Psi_\mu(u) = \calI(u) + \frac{\mu}{2} \norm{u}_2^2. 
\]

\medskip

\noindent
\textbf{II. Constrained minimax values.}
 
 We recall the precise definitions of $\ub$ and $\ob$. 
Under \ref{(g1*)}, there exists $\bar{A}>0$ such that
\begin{equation*}\label{eq_h_Abar}
	\abs{h(s)} \leq 2\bar{A} \abs s, \quad 
	\abs{H(s)} \leq \bar{A} s^2 \quad \text{for all}\ s\in\R. 
\end{equation*}
In \cite[Section 3]{CGIT24}, under \ref{(g0')} and \ref{(g1*)}, % \red{(actually, $\lim_{s\to 0,\pm\infty}{\frac{h(s)}{\abs s^{p-1}s}}=0$ is sufficient)}, 
the map $\zeta_0\in C^2(\R,E)$ is constructed satisfying $\zeta_0\geq 0$ and (recall $\mu=e^{\lambda}$) 
\begin{equation}\label{eq_zeta0_prop}
	\begin{aligned}
		&\Psi_\mu ( \zeta_0 (\lambda) ) < -2 \bar{A} m_1 - 1 \quad 
		\text{for any}\ \lambda \in \R; %\label{3.1}
		\\
		&\frac{1}{2}\norm{ \zeta_0(\lambda) }_2^2 > m_1 \quad 
		\text{for any}\ \lambda \in \R; %\label{3.2}
		\\
		&\max_{0 \leq t \leq 1} I(\lambda, t \zeta_0 (\lambda) ) \to 0 \quad
		\text{as} \ \lambda \to -\infty;
		\\
		&\max_{0 \leq t \leq 1} I(\lambda, t \zeta_0 (\lambda) ) \to -\alpha m_1 %0
\quad
		\text{as} \ \lambda \to +\infty.
	\end{aligned}
\end{equation}
Using $\zeta_0$, we first define $\ub$ by 
\begin{equation*}
	\begin{aligned} 
		&\ub \coloneqq \inf_{\gamma\in \uGamma}\max_{t\in [0,1]} I(\gamma(t)), \\
		&\uGamma \coloneqq \Big\{ 
		\gamma\in C\big([0,1],\RE\big) \;\Big|\; 
		\gamma(0)\in\R\times\{ 0\}, \ I(\gamma(0)) < - 2 \bar A m_1 - 1, \ \gamma (1) \in ({\rm id}\times\zeta_0)(\R) 
		\Big\}.
	\end{aligned} 
\end{equation*}
By construction, the condition on $\gamma (1)$ implies 
$I(\gamma(1)) < -2 \bar{A} m_1 - 1$ and $\frac{1}{2}\norm{\gamma_2(1)}_2^2 > m_1$, and hence 
 \begin{equation*} %\label{3.16}
 \gamma([0,1])\cap %(\R\times 
M_0%)
\not=\emptyset \quad 
 \text{for all}\ \gamma\in\uGamma.
 \end{equation*}

For the second minimax value $\ob$, we define $\gamma_0\in C^2( [0,1] \times \R , \RE )$ by $\gamma_0 (t,\lambda) \coloneqq \left( \lambda , \ t \zeta_0 (\lambda) \right) $ and for $L> 2$, a collar $\calC(L) \subset [0,1] \times \R$ by 
\begin{equation}\label{eq_def_calC(L)}
 \calC(L) \coloneqq \Big( [0,1]\times \big( (-\infty,-L ] \cup [L,\infty) \big) \Big) \bigcup \Big(
 \big([0,1/L]\cup [1-1/L,1]\big)\times \R \Big) \subset [0,1] \times \R. 
\end{equation}
Then we set 
\begin{equation}\label{def_ob_oG}
	\begin{aligned} 
		\ob 
		& \coloneqq \inf_{\gamma\in\oGamma}\sup_{(t,\lambda)\in [0,1]\times\R} I(\gamma(t,\lambda)),\\
		\oGamma & \coloneqq \Big\{ \gamma \in C\big([0,1]\times\R,\RE\big) \;\Big|\;
			\gamma(t,\lambda)=\gamma_0(t,\lambda) \ 
			\text{for}\ (t,\lambda)\in \calC(L_\gamma), %\\
			\; \text{for some} \ L_\gamma>2 %0 % \
		\Big\}.
	\end{aligned}
\end{equation}
We note that $\gamma_0\in\oGamma$ and $\oGamma\not=\emptyset$, 
and it is proved that for any $\gamma \in \oGamma$, 
\begin{equation*}
	\lim_{\lambda \to - \infty} \max_{0 \leq t \leq 1} 
	I\left( \gamma(t, \lambda) \right) = 0, \quad 
	\lim_{\lambda \to +\infty} \max_{0 \leq t \leq 1} 
	I \left( \gamma (t,\lambda) \right) = -\alpha m_1.
\end{equation*}
The minimax values $\ub$ and $\ob$ enjoy the following property 
\begin{equation}\label{eq_general_est_b}
-2 \bar{A} m_1 \leq \ub \leq \min\{-\alpha m_1, 0\} \leq 0 \leq \max\{ -\alpha m_1, 0\} \leq \ob.
\end{equation}

We next recall the notation of Palais-Smale-Pohozaev-Cerami sequences and the related compactness condition 
introduced in \cite{CGIT24}, which will be also used in the proof of \cref{Thm:Exw/og2}.

\begin{Definition}[{\cite[Subsection 5.1]{CGIT24}}]\label{D:PSPC}
	Let $b \in \R$ and set $\scrP_0 \coloneq \big\{ u \in E \ | \ u \geq 0 \big\} $.  
	\begin{enumerate}[label=(\roman*)]
		\item 
		A sequence $(\lambda_j,u_j)_{j=1}^\infty \subset \R \times E$ is called a 
		\emph{Palais-Smale-Pohozaev-Cerami sequence at level $b$} (in short $\PSPC_b$ sequence) if 
		$I(\lambda_j,u_j) \to c$, $\partial_\lambda I(\lambda_j,u_j) \to 0$, $(1+\| u_j \|_E) \| \partial_u I(\lambda_j,u_j) \|_{E^*} \to 0$ 
		and $P(\lambda_j,u_j) \to 0$ where 
		\[
		P(\lambda,u) \coloneq \frac{N-2}{2} \| \nabla u \|_2^2 - N \int_{\RN} G(u) + \frac{e^{\lambda_j}}{2} u^2 \,dx
		\]
		is the Pohozaev functional. 
		
		\item 
		A sequence $(\lambda_j,u_j)_{j=1}^\infty$ is called a $\PSPC_b^*$	sequence if $(\lambda_j,u_j)_{j=1}^\infty$ is a $\PSPC_b$ sequence and satisfies 
		$\distE (u_j,\scrP_0) \to 0$. 
		
		\item 
		The functional $I$ is said to satisfy the \emph{Palais-Smale-Pohozaev-Cerami condition at level $b \in \R$} (in short the $\PSPC_b$ condition) if 
		each $\PSPC_b$ sequence is relatively compact in $\R \times E$. 
		Similarly, $I$ satisfies the $\PSPC_b^*$ condition provided every $\PSPC_b^*$ sequence is relatively compact. 
	\end{enumerate}
\end{Definition}

The following result is proved in \cite[Proposition 5.4]{CGIT24}.

\begin{Proposition}%[{\cite[\red{Proposition 5.4}]{CGIT24}}] 
	\label{prop_compact_Kb} %\label{Proposition:5.4} %\label{Corollary:4.5}
Assume \ref{(g0')} and \ref{(g1*)} with $\alpha=0$. %\ref{(g1)}. 
Then
\begin{itemize}
\item[\rm (i)] for $b<0$, $I$ satisfies the $\PSPC_b$ condition and $K_b$ is compact in $\R \times E$; 
\item[\rm (ii)] for $b>0$, if \ref{(g2)} holds in addition, then $I$ satisfies the $\PSPC_b^*$ condition; as a consequence, 
$K_b^+ \coloneq \big\{ (\lambda,u) \in K_b \mid u \geq 0\big\}$ is compact in $\R \times E$.
\end{itemize}
\end{Proposition}

As a matter of fact, we can give a more precise statement than \cref{thm_old_exist}. This can be easily deduced by \cite[Subsections 4.3, 7.3, 8.5 and Proposition 1.2]{CGIT24} (see also \cite[Proposition 4.6, Proposition 5.4, beginning of Section 7]{CGIT24}).

\begin{Theorem}[\cite{CGIT24}] 
\label{thm_old_ex_precise}
Let $N\geq 2$, $m=m_1$ be given in \eqref{eq_m1}, and assume \ref{(g0')}, %\ref{(g0)}, 
\ref{(g1*)} with $\alpha=0$. %--\ref{(g1)}.
Then \eqref{eq_main} with $m=m_1$ has a positive radially symmetric solution 
$(\mu,u)\in (0,\infty)\times H^1(\R^N)$ if one of the following conditions is satisfied: 
\begin{itemize}
\item $\ub<0$,
\item $\ob>0$ if $N=2,3,4$ or $\ob > 0$ and \ref{(g2)} (or $N \geq 5$ and \ref{(g3)}) hold,
\item $\ub=0=\ob$.
\end{itemize}
\end{Theorem}

\medskip
\noindent
\textbf{III. Unconstrained minimax values, optimal paths and asymptotic behaviors of ground states.}

For $\lambda\in\R$ we introduce the mountain pass value $b(\lambda)$ of the functional (recall $\mu=e^{\lambda}$) 
$ u\mapsto I(\lambda,u)$ by
\begin{equation}\label{eq_def_blambda}
\begin{aligned}
 & b(\lambda) \coloneqq \inf_{\lambda\in\widehat{\Lambda}_\mu} \max_{t\in [0,1]}
 I(\lambda,\gamma(t)), \\
 & \widehat{\Lambda}_\mu \coloneqq \big\{ \gamma \in C([0,1],E) \mid \gamma(0)=0, \ 
 \gamma(1) = \zeta_0(\lambda) \big\}. 
\end{aligned}
\end{equation}
We have %(recall $\mu=e^{\lambda}$)
	\begin{equation} \label{eq_amu_blambda} %\label{3.26}
		b(\lambda) = a(\mu)-\mu m_1, 
	\end{equation}
where $a(\mu)$ is the mountain pass value of the functional $\Psi_\mu$:
\begin{equation}\label{eq_def_amu}
\begin{aligned} %\label{3.8}
 &a(\mu) \coloneqq \inf_{\gamma\in \Lambda_\mu}\max_{t\in [0,1]} \Psi_\mu(\gamma(t)), \\
 %\label{3.9}
 &\Lambda_\mu \coloneqq \big\{ \gamma\in C([0,1],E) \mid \gamma(0)=0, \
 \Psi_\mu(\gamma(1)) <0 \big\}.
\end{aligned}
\end{equation}
Furthermore, under \ref{(g0')} and \ref{(g1*)}, the key relation \eqref{eq_three_b} holds, that is,
\begin{equation}\label{eq_b_blambda} %\label{Proposition:4.2}
	\ub \leq b(\lambda) \leq \ob \quad \text{for each $\lambda \in \R$.}
\end{equation}
As a consequence, $\tilde{b}$ defined in \eqref{eq_def_btilde} satisfies $\ub \leq \tilde{b} \leq \ob$.
Remark that the proof of \cite[Proposition 4.2 (i) and (ii)]{CGIT24} works under \ref{(g1*)}. 

We next discuss the existence of optimal paths and recall the following results proved in \cite{JT0} (see also \cite[Section A.2]{CGIT24}).

\begin{Proposition}
	\label{prop_paths_f} % \label{Proposition:A.2} %\label{Proposition:A.6}
	Assume $N\geq 2$ and $f\in C(\R)$ satisfies 
	\begin{enumerate}[label={\rm (f\arabic*)}]
		\setcounter{enumi}{0}%-1}
		\item \label{(f1)}
		$f$ has subcritical growth at $s=\pm\infty$, that is,
		\begin{itemize}
			\item[{\rm (1)}] for $N\geq 3$,
			$\displaystyle \lim_{\abs{s} \to \infty} {\frac{f(s)}{\abs s^{2^*-1}}} = 0$,
			\item[{\rm (2)}] for $N=2$
			$\displaystyle \lim_{\abs{s} \to \infty}{\frac{f(s)}{e^{\alpha s^2}}} = 0 \quad
			\text{for any}\ \alpha>0$;
		\end{itemize}
		\item \label{(f2)}
		$\displaystyle \lim_{s\to 0}{\frac{f(s)}{s}} \in (-\infty, 0)%<0
$;
		\item \label{(f3)}
		there exists $s_0>0$ such that $F(s_0) > 0$ where 
		$\displaystyle F(s) \coloneqq \int_0^s f(\tau)\,d\tau$. 
	\end{enumerate}
	Define the functional $L: E\to \R$ by 
	\begin{equation*}
		L(u) \coloneqq \half\norm{\nabla u}_2^2 -\intRN F(u)\,dx.
	\end{equation*}
	Suppose that $u_0\in E\setminus\{ 0\}$ satisfies the Pohozaev identity
	\begin{equation} \label{eq_paths_g_1} %\label{A.28}
		{\frac{N-2}{2}}\norm{\nabla u_0}_2^2 -N\intRN F(u_0)\,dx =0
	\end{equation}
	and when $N = 2$, we additionally suppose that
	\begin{equation} \label{eq_paths_g_2} %\label{A.29}
		\intRN f(u_0)u_0\, dx > 0.			
	\end{equation}
	Then by writing $u_{0,\theta} \coloneqq u_0(\cdot/\theta)$, the following hold: 
	
	\begin{enumerate}
		\item[\rm (i)] 
		when $N\geq 3$, for a large $T>1$, the %a 
path $\gamma_0 \in C([0,1],E)$ defined by
		\[ 
		\gamma_0(t)(x) \coloneq \begin{dcases}
			u_{0 , Tt} &\text{for}\ t\in (0,1],\\
			0 &\text{for}\ t=0,
		\end{dcases}
		\]
		satisfies 
		\[
		u_0 \in \gamma_0([0,1]), \quad L(\gamma_0(1)) \ll -1, \quad L(\gamma_0(t)) < L(u_0) \quad \text{for every $t \in [0,1] \setminus \{1/T\}$}; 
		\]
		\item[\rm (ii)] 
		when $N=2$, there exist $0< \theta_0 \ll 1$, $1 \ll \theta_1$ and $t_1>1$ close to $1$ such that 
		the %a 
path $\gamma_0 \in C([0,1] , E)$ defined by joining the following three paths 
		\[	{\rm (a)}\ [0,1]\to E;\, t\mapsto tu_{0,\theta_0}, \quad 
		{\rm (b)}\ [\theta_0,\theta_1]\to E;\, \theta\mapsto u_{0,\theta}, \quad
		{\rm (c)}\ [1,t_1]\to E;\, t\mapsto tu_{0,\theta_1}
		\]
		satisfies 
		\[
		\begin{aligned}
			&u_0 \in \gamma_0([0,1]), \quad 
			L(\gamma_0(1)) \ll -1, \quad \max_{t \in [0,1]} L(\gamma_0(t)) = L(u_0), \\
			&
			L(\gamma_0(t)) = L(u_0) \ \text{if and only if} \ \text{$\gamma_0(t) = u_{0,\theta}$ holds for some $\theta \in [\theta_0,\theta_1]$}. 
		\end{aligned}
		\]
	\end{enumerate}
\end{Proposition}

\begin{Remark}\label{R:opt-path}
		According to \cite{BeLi83I}, under 
		$f \in C(\R)$ with \ref{(f1)}--\ref{(f3)}, any critical point of $L$ satisfies \eqref{eq_paths_g_1} and \eqref{eq_paths_g_2}. 
\end{Remark}

 We conclude by an asymptotic behavior of ground states of the unconstrained problem (see also \cite{JZZ24} for similar properties).
\begin{Lemma}\label{lem_asympt_behav}
	Let $\mu>0$, 
	$f \in C(\R)$ and $f(t) \coloneqq g(t)-\mu t$ be odd and satisfy \ref{(f1)}--\ref{(f3)}. 
	Then there exists a ground state solution $w_\mu$ of \eqref{eq:Psi_mu} such that 
	\begin{equation} \label{eq_wmu_explode} %\label{A.31}
		a(\mu) = \Psi_\mu (w_\mu), \quad w_\mu > 0 \quad \text{in} \ \RN, \quad 
		\norm{w_\mu}_\infty\to\infty \quad \text{as}\ \mu\to\infty.
	\end{equation}
\end{Lemma}

\begin{proof}
From \cite{BeGaKa83,BeLi83I,JT0}, \eqref{eq:Psi_mu} admits a ground state solution $w_\mu \in C^2(\RN)\cap E$ with $a(\mu) = \Psi_\mu(w_\mu)$ 
and $w_\mu > 0$ in $\RN$. 
Let $x_\mu \in \RN$ satisfy $w_\mu(x_\mu) = \norm{w_\mu}_\infty$. From $-\Delta w_\mu(x_\mu) \geq 0$ it follows that 
\[
g \left( w_\mu(x_\mu) \right) = -\Delta w_\mu (x_\mu) + \mu w_\mu(x_\mu) \geq \mu w_\mu(x_\mu) > 0,
\]
which yields $g(w_\mu(x_\mu)) / w_\mu(x_\mu) \to \infty$ as $\mu \to \infty$. Hence, $w_\mu(x_\mu) = \| w_\mu \|_\infty \to \infty$, 
and \eqref{eq_wmu_explode} follows. 
\end{proof}

\medskip
\noindent
\textbf{IV. Power case.}

Recall first the following relation derived from the Gagliardo-Nirenberg inequality \cite[Corollary 2.2]{CGIT24}: 
 \begin{equation} \label{eq_estimate_GN}
 \half\norm{\nabla u}_2^2 \geq {\frac{1}{p+1}}\norm u_{p+1}^{p+1} 
 \qquad \text{for all}\ u\in E\ \text{with}\ \frac{1}{2}\norm u_2^2 \leq m_1.
 \end{equation}
 We set
\[
G_0(s) \coloneqq \frac{1}{p+1} |s|^{p+1}
\]
so that $G= G_0 + H$. Let $\omega_\mu \in E$ be the unique positive solution of 
\[
-\Delta \omega + \mu \omega =\abs \omega^{p-1}\omega \quad \text{in}\ \R^N,
\]
where $\mu>0$ is fixed.
We have 
\begin{equation}\label{eq_omegamu}
\omega_\mu = \mu^{N/4} \omega_1 (\mu^{1/2} \cdot)
\end{equation}
 and
 \begin{equation} \label{eq_scaling_omega} %\label{2.9}
 \begin{aligned} 
 &\norm{\omega_\mu}_2^2=2m_1, \quad \norm{\nabla \omega_\mu}_2^2=\mu N m_1, \quad
 \norm{\omega_\mu}_{p+1}^{p+1}=\mu (N+2)m_1, \\
 &\Psi_{0\mu}(\omega_\mu)=\mu m_1, \quad \Psi_{0\mu}'(\omega_\mu)=0, 
 \end{aligned}
 \end{equation}
 where $\Psi_{0\mu}: E\to\R$ is the corresponding action functional
\begin{equation*} %\label{2.8}
\Psi_{0\mu}(u) \coloneqq \half\norm{\nabla u}_2^2 +{\frac{\mu}{2}}\norm u_2^2 -{\frac{1}{p+1}}\norm u_{p+1}^{p+1}. 
\end{equation*}
Introduce also the energy functional $\calI_0 : E %\calS_0
\to \R$ and the Lagrangian functional
$I_0: \R \times E \to \R$
	\begin{equation*}
			\calI_0(u) \coloneqq \half\norm{\nabla u}_2^2-{\frac{1}{p+1}}\norm u_{p+1}^{p+1}, \qquad
			 I_0(\lambda,u) \coloneqq \half\norm{\nabla u}_2^2-{\frac{1}{p+1}}\norm u_{p+1}^{p+1} 
 + e^\lambda \left( \frac{1}{2} \norm{u}_{2}^2 - m_1 \right).
	\end{equation*}
Calling $\ub_0$, $\ob_0$ and $b_0(\lambda)$ the values $\ub$, $\ob$ and $b(\lambda)$ when $G$ is substituted by $G_0$, we obtain \cite[Lemma 4.4]{CGIT24}
 \begin{equation}\label{eq_values_b0} %\label{Lemma:4.4} %\label{Lemma:3.9}
 \ub_0 = \ob_0 =0 \quad \text{and}\quad b_0(\lambda)=0 \ \text{for all}\ \lambda\in\R.
 \end{equation}
Finally 
since $I(\lambda,u) = \Psi_{0\mu} (u) - \int_{\RN} H(\omega_\mu) \, dx - \mu m_1$ and $H(s) = o(|s|^{p+1})$ as $s \to 0$ due to \ref{(g1*)}, 
it follows from \eqref{eq_scaling_omega} that 
\begin{equation}\label{eq_lim_lam_omegamu} %\label{Lemma:3.3}
\lim_{\lambda\to -\infty}I(\lambda,\omega_\mu)=0.
\end{equation}

%%%%%%%%%%%%%%%%%%%%%%%%%%%%%%%%%%%%%%%%%%%%%%%%%%%%%%%%%%%%%%%%%%
\section{Examples of $g$ for $\ub < 0$ and $\ob > 0$: proof of Theorem \ref{thm_examples}} % \cref{thm_examples}} 
\label{Section_examples} %\label{Subsection:A.4}

In this subsection, we give some examples of $g$ for which 
either $\ub < 0=\ob$ or else $\ub=0<\ob$ holds; in the first case, we can ensure the existence of an $L^2$-minimum. 
Furthermore, we also provide an example of $g$ satisfying \ref{(g0')}, \ref{(g1*)} with $\alpha=0$, %\ref{(g1)}, 
\ref{(g3)} and $\ub < 0 < \ob$, and hence 
we may find at least two solutions of \eqref{eq_main} with $m=m_1$ for such a $g$ 
(see \cref{thm_old_ex_precise}).

%%%%%%%%%%%%%%%%%%%%%%%%
\subsection{Case $\ub=0$ or $\ob=0$}

We start with simple examples. Recall $G_0(s)=\frac{1}{p+1}|s|^{p+1}$ and relation \eqref{eq_b_blambda}.

\begin{Proposition} \label{prop_b_neg_post} % \label{Example:A.8}
Assume \ref{(g0')} and \ref{(g1*)} with $\alpha=0$. %\ref{(g1)}. 
The following facts hold.
	\begin{enumerate}[label={\rm (\roman*)}]
		\item Assume %\ref{(g1)} and 
		\begin{equation*}
			\begin{aligned}
				&G(s) \geq G_0(s) \quad \text{for all}\ s\in\R, 
				\quad G(s_0) > G_0(s_0) \quad \text{for some}\ s_0\in\R.
			\end{aligned}
		\end{equation*}
		Then 
		\begin{equation*}
			b(\lambda)<0 \quad \text{for large}\ \lambda \quad \text{and} \quad 
			\ub <0 = \ob.
		\end{equation*}
		\item Assume %\ref{(g1)} and 
		\begin{equation*}
			\begin{aligned}
				&G(s) \leq G_0(s) \quad \text{for all}\ s\in\R, 
				\quad 
				G(s_0) < G_0(s_0) \quad \text{for some}\ s_0\in\R.
			\end{aligned}
		\end{equation*}
		Then 
		\begin{equation*}
			b(\lambda)>0 \quad \text{for large}\ \lambda \quad \text{and} \quad 
			\ub =0 < \ob.
		\end{equation*}		
	\end{enumerate}
\end{Proposition}

Examples of $G$ satisfying the abovementioned conditions together with \ref{(g3)} are straightforward: it is sufficient to consider
 $G(s)= A%a
(s) G_0(s)$ where 
\[A(s)=
\begin{dcases} 
	1 \pm \varepsilon e^{\frac{1}{(\abs{s}-2)^2-1}} & \hbox{for $\abs{s} \in (1,3)$}, \\ 
	1 &\hbox{elsewhere},
\end{dcases}\]
 and $\varepsilon >0$ is small enough. Notice anyway that, if $\ub <0$, then by \cref{thm_old_ex_precise}, assumption \ref{(g3)} is not needed to get the existence of a solution and thus the choice of $G$ is much more free.
 
\begin{proof}[Proof of \cref{prop_b_neg_post}]
	We show (ii) since the proof of the other case is similar. 
Recalling \eqref{eq_b_blambda}, it
suffices to show $b(\lambda)>0$ for 
	large $\lambda\in\R$. 
By \cref{lem_asympt_behav}, let $w_\mu \in E$ be a ground state of \eqref{eq:Psi_mu} with $\mu=e^{\lambda}$, 
and choose a large $\lambda>1$ such that $G \not \equiv G_0$ on 
	$[0,\norm{w_\mu}_\infty]$. This choice yields 
	\begin{equation} \label{eq_dim_Gwmu} % \label{A.32}
		\intRN G \left( w_\mu \left( x/ \theta \right) \right) dx < \intRN G_0 \left( w_\mu (x / \theta) \right) dx 
		\quad \text{for any $\theta > 0$}.
	\end{equation}
By \cref{prop_paths_f} and \cref{lem_asympt_behav}, there exists a path $\gamma_0 \in \Lambda_\mu$ 
such that for some $0< \theta_0 < 1 < \theta_1$, 
\begin{equation*}
	\begin{aligned}
		&\max_{t\in [0,1]}\Psi_\mu(\gamma_0(t)) = \Psi_\mu(w_\mu)=a(e^\lambda),
		\\
		&\Psi_\mu(\gamma_0(t))=a(e^\lambda) \quad \text{if and only if} \quad 
		\begin{dcases}	\gamma_0(t)=w_\mu	&\text{when}\ N\geq 3,\\
			\gamma_0(t)\in \big\{w_\mu(\cdot/\theta) \mid \theta\in[\theta_0,\theta_1]\big\}
			&\text{when}\ N=2.
		\end{dcases}
	\end{aligned}
\end{equation*}
Now we can show
\begin{equation} \label{eq_dim_psi0_a} %\label{A.33}
	\Psi_{0\mu}(\gamma_0(t)) < a(e^\lambda) \quad \text{for all}\ t \in [0,1].
\end{equation}
	In fact, when $\Psi_\mu(\gamma_0(t))=a(e^\lambda)$, we have 
	$\gamma_0(t)=w_\mu(\cdot/\theta)$ ($\theta=1$ or $\theta\in [\theta_0,\theta_{1}%2
]$) 
	and by \eqref{eq_dim_Gwmu},
	\begin{equation*}
		\Psi_{0\mu}(\gamma_0(t))=\Psi_{0\mu}(w_\mu(\cdot/\theta)) 
		< \Psi_\mu(w_\mu(\cdot/\theta)) =a(e^\lambda).
	\end{equation*}
	When $\Psi_\mu(\gamma_0(t))<a(e^\lambda)$, it is immediate to see that 
	\begin{equation*}
		\Psi_{0\mu}(\gamma_0(t))\leq \Psi_\mu(\gamma_0(t)) <a(e^\lambda).
	\end{equation*}
Thus \eqref{eq_dim_psi0_a} holds. 
Since $\omega_{\mu}$ is obtained via the mountain pass theorem and $\gamma_0 \in \Lambda_\mu$, 
by \eqref{eq_scaling_omega} and $\Psi_{0\mu} (\gamma_0(1)) \leq \Psi_\mu(\gamma_0(1)) < 0$, 
we obtain 
\begin{equation*}
	 \mu m_1 = \Psi_{0\mu} (\omega_\mu) 
	\leq \max_{t\in [0,1]} \Psi_{0\mu}(\gamma_0(t)) < a(e^\lambda),
\end{equation*}
	which implies (recall \eqref{eq_amu_blambda} and \eqref{eq_values_b0})
	\begin{equation*}
		0=b_0(\lambda)< b(\lambda).
	\end{equation*}
Therefore, \eqref{eq_b_blambda} yields $\ob > 0$. 
We remark that $\ub \geq 0$ follows from \eqref{eq_values_b0}, $I_0(\lambda,u) \leq I(\lambda,u)$ and $0 = b_0(\lambda) \leq b(\lambda)$ for all $\lambda \in \R$. 
On the other hand, since $\ub \leq 0$ % $b(\lambda) \leq 0$ 
holds by \eqref{eq_general_est_b}, we have $\ub = 0$. 
\end{proof}

%%%%%%%%%%%%%%%%%%%%%%%%
\subsection{Case $\ub < 0 < \ob$}

	Finally we provide an example of $g$ satisfying \ref{(g0')}, \ref{(g1*)} with $\alpha=0$, %\ref{(g1)}, 
	\ref{(g3)} and $\ub < 0 < \ob$. 
To this end, let us consider 
\begin{equation}\label{def:HG-a}
	H(s) \coloneqq \frac{A(s)}{p+1} |s|^{p+1}, \quad G(s) = \frac{1}{p+1}|s|^{p+1} + H(s) = \frac{1+A(s)}{p+1} |s|^{p+1},
\end{equation}
where 
$A \in C^1(\R)$ 
is an even function and satisfies the following conditions: for $\alpha \in \left[ - \frac{1}{2} , \frac{1}{2} \right]$ and $L > 1$, 
	\begin{align}
		&-\frac{1}{2} \leq A(s) \leq \frac{1}{2} \quad \text{for all $s \in [0,\infty)$}, \label{eq_propr_a_1} %\label{A.34}
		\\
		&\text{$A(s) = 0$ in neighborhoods of $0$ and $+\infty$}, \label{eq_propr_a_2} %\label{A.35}
		\\
		&
		\left|A'(s) s\right| < \frac{1}{N^2},
		\label{eq_propr_a_3} %\label{A.36}
		\\
		&A(s) = \alpha \quad \text{for all $s \in \left[ \frac{1}{L}, \, L \right]$}. \label{eq_propr_a_4} %\label{A.37}
	\end{align}

We begin to prove the existence of $A$ satisfying \eqref{eq_propr_a_1}--\eqref{eq_propr_a_4}. 

\begin{Lemma} \label{lem_exist_a} %\label{Lemma:A.9}
	Let $L > 1$ and $\alpha \in \left[ - \frac{1}{2} , \frac{1}{2} \right]$. 
	If $ 0 \leq \alpha \leq \frac{1}{2}$ (resp. $- \frac{1}{2} \leq \alpha \leq 0$), 
	then there exists an even non-negative (resp. non-positive) function $A(\cdot)$ satisfying \eqref{eq_propr_a_1}--\eqref{eq_propr_a_4}. 
\end{Lemma}

\begin{proof}
Write $\tau \coloneqq 1/(2N^2)$ and set 
\[
A_0(s) 
\coloneqq \begin{dcases}
	\alpha + \left( \sgn \alpha \right) \tau 
	\big( \log (L+1) + \log s \big) &\text{for} \ s \in \left[ \exp \left( - \frac{\abs{\alpha}}{\tau} - \log (L+1) \right), \, \frac{1}{L+1} \right],
	\\
	\alpha & \text{for} \ s\in \left[ \frac{1}{L+1} , \, L+1 \right], 
	\\
	\alpha + \left(\sgn \alpha\right) \tau \left( - \log s + \log \big(L+1\right) \big)
	& \text{for} \ s \in \left[L+1, \, \exp \left( \frac{\abs{\alpha}}{\tau} + \log \left(L+1\right) \right)\right],
\end{dcases}
\]
and $A_0(s)\coloneqq 0$ for any other $s\geq 0$; for $s \leq 0$, we extend %$a_0(s)$ as 
$A_0(s) \coloneqq A_0(-s)$. 
Notice that $A_0$ is a Lipschitz function and satisfies \eqref{eq_propr_a_1}, \eqref{eq_propr_a_2}, \eqref{eq_propr_a_4} on $[1/(L+1), L+1]$ 
and $|A_0'(s) s| \leq \tau $ at $s$ where $A_0$ is differentiable. 
After mollifying $A_0$, we get $A \in C^1(\R)$ enjoying also \eqref{eq_propr_a_3}. % and \eqref{eq_propr_a_4}. 
\end{proof}

We next prove that $G$ with $g \coloneq G'$ in \eqref{def:HG-a} has the following properties. 

\begin{Lemma}\label{Lem:AR}
	Let $A \in C^1(\R)$ be an even function enjoying \eqref{eq_propr_a_1}--\eqref{eq_propr_a_4}. 
	Then $g$ in \eqref{def:HG-a} satisfies \ref{(g0')}, \ref{(g1*)} with $\alpha=0$, \ref{(g3)} and the Ambrosetti-Rabinowitz condition: 
	\[
	g(s)s \geq \theta %\frac{p+3}{2} 
G(s) > 0 \quad \text{for any $s \in \R \setminus \{0\}$},
	\]
where $\theta=\frac{2(N^2+2N-1)}{N^2} \in (2, p+1)$.
\end{Lemma}

From the proof below, 
the constant $\theta=\frac{2(N^2+2N-1)}{N^2}$ in \cref{Lem:AR} may be improved and approaches $p+1$ 
by changing $\frac{1}{N^2}$ in \eqref{eq_propr_a_3} to smaller constants. 

\begin{proof}[Proof of \cref{Lem:AR}]
Conditions \ref{(g0')} and \ref{(g1*)} with $\alpha = 0$ clearly hold by \eqref{def:HG-a} and \eqref{eq_propr_a_2}. 
On the other hand, direct computations, \eqref{eq_propr_a_1} and \eqref{eq_propr_a_3} lead to 
\[
\begin{aligned}
	G(s) - \frac{N-2}{2N} g(s) s 
	&= 
	\left( \frac{2}{N(N+2)} \big(1+A(s)\big) - \frac{N-2}{4(N+2)} A'(s) s \right) \abs{s}^{p+1}
	\\
	&\geq \left( \frac{1}{N(N+2)} - \frac{N-2}{4(N+2)} \frac{1}{N^2} \right) \abs{s}^{p+1} 
	= \frac{1}{N(N+2)} \left( 1 - \frac{N-2}{4N} \right) \abs{s}^{p+1} \geq 0.
\end{aligned}
\]
Thus, \ref{(g3)} holds. 
Finally, from \eqref{eq_propr_a_1}, \eqref{eq_propr_a_3} and 
\[
g(s) s - \theta G(s) = \left( \frac{1}{p+1} A'(s) s + \left( 1 - \frac{\theta}{p+1} \right) \big(1+A(s)\big) \right) \abs{s}^{p+1} 
\geq \frac{1}{2} \left( 1 - \frac{\theta}{p+1} - \frac{2}{N^2(p+1)} \right) |s|^{p+1},
\]
a straightforward computation together with the definition of $\theta$  implies the Ambrosetti-Rabinowitz condition. 
The proof is completed. 
\end{proof}

For an even function $A $ with \eqref{eq_propr_a_1}--\eqref{eq_propr_a_4}, define $I(A ; \cdot, \cdot): \R \times E \to \R$ by (recall $\mu=e^{\lambda}$)
	\[
		I(A;\lambda, u) \coloneqq 
		\frac{1}{2} \| \nabla u \|_2^2 - \frac{1}{p+1} \int_{\RN} \big( 1 + A(u) \big) |u|^{p+1} \, dx 
		+ \mu \left( \frac{1}{2} \| u \|_2^2 - m_1 \right)
	\]
and write $b(A;\lambda)$ for the mountain pass value of $u \mapsto I(A;\lambda, u)$. 
For $u \in E$ and $\lambda \in \R$, we introduce 
	\[
		u_\lambda \coloneqq \mu^{N/4} u \big( \mu^{1/2} \cdot \big). 
	\]
Then it is easily seen that 
	\[
		\begin{aligned}
			I(A; \lambda, u_\lambda) 
			&= 
			\mu \left( \frac{1}{2} \| \nabla u \|_2^2 + \frac{1}{2} \| u \|_2^2 
			- \frac{1}{p+1} \int_{\RN} \left( 1 + A \big( \mu^{N/4} u \big) \right) |u|^{p+1} \, dx - m_1
			\right)
			\\
			&\eqqcolon \mu \left( K(A;\lambda, u) - m_1 \right) . 
		\end{aligned}
	\]
Let $\beta(A;\lambda)$ stand for the mountain pass value of $u \mapsto K(A;\lambda,u)$. 
From \eqref{eq_b_blambda} and $b(A;\lambda) = \mu ( \beta (A;\lambda) - m_1 ) $, 
it is enough to find $A$, $\lambda_1$ and $\lambda_2$ such that 
	\begin{equation} \label{eq_beta_12} %\label{A.38}
		\beta (A; \lambda_1) < m_1 < \beta (A;\lambda_2). 
	\end{equation}
	To find such $\lambda_1$ and $\lambda_2$, we remark that for $\alpha \in [-\frac{1}{2} , \frac{1}{2}]$, a least energy solution of 
	\begin{equation} \label{eq_schr_alpha} %\label{A.39}
		-\Delta u + u = (1+\alpha) |u|^{p-1} u \quad \text{in} \ \RN
	\end{equation}
is given by $(1+\alpha)^{-1/(p-1)} \omega_1 $. Thus, the mountain pass value of the functional corresponding to \eqref{eq_schr_alpha} is 
	\[
		(1+\alpha)^{ - 2/ (p-1) } m_1. 
	\]

	To proceed, we need the following result.
	\begin{Lemma} \label{lem_Leps_alpha} %\label{Lemma:A.10}
		Let $\alpha \in \left[ -\frac{1}{2}, \frac{1}{2} \right]$. 
		For any $\varepsilon > 0$ there exists $L_\varepsilon > 1$ such that for each $L \geq L_\varepsilon$ and 
		even function $A \in C^1(\R)$ with \eqref{eq_propr_a_1}--\eqref{eq_propr_a_4}, the value $\beta(A;0)$ satisfies 
		\[
			\left| \beta(A;0) - \left( 1 + \alpha \right)^{-2/(p-1)} m_1 \right| < \varepsilon. 
		\]
	\end{Lemma}

	\begin{proof}
We only give a sketch of a proof since this lemma is standard. 
Let $\alpha \in [-1/2, 1/2]$ be given and $(L_n)_{n\in \N}$, $(\alpha_n)_{n\in \N}$ and $(A_n)_{n\in \N}$ 
be arbitrary sequences satisfying $L_n \to \infty$, $\alpha_n \to \alpha$, and $A_n$ is an even function 
satisfying \eqref{eq_propr_a_1}--\eqref{eq_propr_a_4} with $L_n$ and $\alpha_n$. 
It suffices to show $\beta (A_n;0) \to (1+\alpha)^{-2/(p-1)} m_1$. 
By \eqref{eq_propr_a_1}, we see 
\[
	\beta \left( \frac{1}{2} ; 0 \right) \leq \beta(A_n;0) \leq \beta \left( - \frac{1}{2} ; 0 \right) \quad \text{for all $n \geq1$}.
\] 
In particular, $\beta(A_n;0)$ is bounded. Let $u_n \in E$ be a critical point of $K( A_n; 0,\cdot )$ with 
$K(A_n;0,u_n) = \beta(A_n;0)$. 
According to \cref{Lem:AR}, every $G_n = (1+A_n(s)) \abs{s}^{p+1}/(p+1)$ with $g_n = G_n'$ 
satisfies the Ambrosetti-Rabinowitz condition with the fixed constant $\theta $. Therefore, the boundedness of $K(A_n;0,u_n)$ 
implies that $(u_n)_{n\in \N}$ is bounded in $E$. Using Radial lemma and 
\[
\frac{g_n(s)}{|s|^{p-1}s} \to 1 + \alpha \quad \text{in $L^\infty_{loc}\big( [-k,k] \setminus \{0\} \big)$ for each $k \in \N$}, 
\]
by choosing suitable signs $\pm$ and extracting a subsequence, 
we may check that $\pm u_{n_k} \to (1+\alpha)^{-1/(p-1)} \omega_1$ strongly in $E$ 
since \eqref{eq_schr_alpha} has a unique positive solution in $E$. 
Thus, $\beta (A_{n_k}%_n
;0) \to (1+\alpha)^{-2/(p-1)} m_1$. 
Since the limit does not depend on choices of subsequences, we have 
$\beta (A_n;0) \to (1+\alpha)^{-2/(p-1)} m_1$. 
\end{proof}

	Now we find $a$ and $\lambda_1,\lambda_2$ satisfying \eqref{eq_beta_12}. 
Choose $\alpha_1,\alpha_2$ and $\varepsilon$ so that 
	\[
		\begin{aligned}
			-\frac{1}{2} < \alpha_1 < 0 < \alpha_2 < \frac{1}{2}, \quad \varepsilon > 0, \quad 
			(1+\alpha_2)^{ - 2 /(p-1) } m_1 + \varepsilon < m_1 < (1+\alpha_1)^{ - 2/(p-1) } m_1 - \varepsilon. 
		\end{aligned}
	\]
Fix $L_\varepsilon > 1$ so that \cref{lem_Leps_alpha} holds for $\alpha = \alpha_1, \alpha_2$. 
By \cref{lem_exist_a}, we may find even functions $A_1,A_2 \in C^1(\R)$ satisfying 
\eqref{eq_propr_a_1}--\eqref{eq_propr_a_4} with $L=L_\varepsilon$ and $\alpha_i$. 
For sufficiently large $\ell > 0$, \eqref{eq_propr_a_2} gives 
	\begin{equation}\label{eq_supp_a1} %\label{A.40}
		\supp A_1 (\cdot) \cap \supp A_2 \big( e^{ - N\ell/4 } \cdot \big) = \emptyset. 
	\end{equation}
For such an $\ell$, we set 
	\[
		A_\ell \coloneqq A_1 + A_2 \big( e^{ -N \ell /4 } \cdot \big). 
	\]
Remark that if $A$ satisfies \eqref{eq_propr_a_1}--\eqref{eq_propr_a_3}, then $A(\nu \cdot)$ also satisfies \eqref{eq_propr_a_1}--\eqref{eq_propr_a_3} for all $\nu > 0$. 
Therefore, thanks to \eqref{eq_supp_a1}, $A_\ell$ and $A_\ell( e^{N\ell/ 4} \cdot)$ satisfy \eqref{eq_propr_a_1}--\eqref{eq_propr_a_3} and 
	\[
		A_\ell (s) = \alpha_1 \quad \text{for each $s \in \left[ \frac{1}{L_\varepsilon} , L_\varepsilon \right]$}, \quad 
		A_\ell \big( e^{N\ell/4} s \big) = \alpha_2 \quad \text{for every $s \in \left[ \frac{1}{L_\varepsilon} , L_\varepsilon \right]$}. 
	\]
Recalling 
	\[
		\begin{aligned}
			K( A_\ell , 0 , u) &= \int_{\RN} \frac{1}{2} |\nabla u|^2 + \frac{1}{2} |u|^2 - \frac{1}{p+1} \big(1+A_\ell (u) \big) |u|^{p+1} \, dx,
			\\
			K(A_\ell , \ell , u) &= \int_{\RN} \frac{1}{2} |\nabla u|^2 
			+ \frac{1}{2} |u|^2 - \frac{1}{p+1} \left( 1+A_\ell \big( e^{ N\ell /4 } u \big) \right) |u|^{p+1} \, dx ,
		\end{aligned}
	\]
we infer from \cref{lem_Leps_alpha} and the choice of $\alpha_1 ,\alpha_2$ and $\varepsilon$ that 
	\[
		\beta (A_\ell ; 0 ) > \left( 1 + \alpha_1 \right)^{ - 2/(p-1) } - \varepsilon > m_1 > 
		\left(1 + \alpha_2 \right)^{ - 2/(p-1) } + \varepsilon 
		> \beta \left( A_\ell \big( e^{N\ell/4} \cdot \big), 0 \right) = \beta \left( A_\ell ; \ell \right).
	\]
Hence, \eqref{eq_beta_12} holds and we may find the desired example. 

\begin{proof}[Proof of \cref{thm_examples}]
The three cases of \cref{thm_examples} follow by \cref{prop_b_neg_post} (i) and (ii) (see also the subsequent comment), example \eqref{def:HG-a} and the subsequent construction, %the construction made in this \cref{subsec_ub<0<ob}, 
and the model case $g_0(t)=|t|^{1+\frac{4}{N}}t$, see \eqref{eq_values_b0}. The existence of two solutions finally follows by \cref{thm_old_ex_precise}. 
\end{proof}

%%%%%%%%%%%%%%%%%%%%%%%%%%%%%%%%%%%%%%%%%%%%%%%%%%%%%%%%%%%%%%%%%%
\section{$L^2$-minimization problem}
\label{Section_L2min} %\label{Section:7} 

In this section, we study the attainability of $d$ and relate $d$ to $\ub$, where 
$d$ and $\calI$ are defined in \eqref{eq_min_L2} and \eqref{eq_def_energyf}. 

%%%%%%%%%%%%%%%%%%%%%%%%
\subsection{The values $\underline{b}$ and $d$: proof of Theorem \ref{thm_exists_min}} %\cref{thm_exists_min}} 
\label{subsec_theorem_exist_min}

Throughout this subsection, conditions \ref{(g0')} and \ref{(g1*)} are always supposed to hold. 

We note that $\gamma([0,1])\cap M_0\not=\emptyset$ for all $\gamma\in\uGamma$ and thus
	\begin{equation} \label{eq_b_d_geq} %\label{7.1}
		\ub \geq \inf_{(\lambda,u)\in M_0} I(\lambda,u)=\inf_{u\in \calS_0}\calI(u)=d.
	\end{equation}
To prove \cref{thm_exists_min} (i), it suffices to show $\ub\leq d$. 
To achieve this, we introduce another minimizing value $d_+$ by
	\begin{equation*}
			d_+\coloneqq \inf_{u\in \calS_{0,+}} \calI(u),
	\end{equation*}
where
	\begin{equation*}
			\calS_{0,+} \coloneqq \begin{cases}
			\calS_0						&\text{when}\ N\geq 3,\\
			\big\{ u\in\calS_0 \mid Q(u)>0 \big\}	&\text{when}\ N=2,
		\end{cases}
	\end{equation*}
and
	\begin{equation*}
			Q(u)\coloneqq \intRN g(u)u-2G(u)\, dx.
	\end{equation*}
Clearly
	\begin{equation} \label{eq_d+_d} % \label{7.2}
		d_+ \begin{cases}	\geq d &\text{for}\ N=2,\\ =d &\text{for}\ N\geq 3.
		\end{cases}
	\end{equation}
Since $g(s)s-2G(s)\sim (1-{\frac{2}{p+1}})\abs s^{p+1}$ as $s\sim 0$, 
$Q(\omega_\mu)>0$ holds for small $\mu>0$ (recall \eqref{eq_omegamu}). 
Remark also that $\lim_{\mu\to 0}\calI(\omega_\mu)=\lim_{\lambda\to -\infty}I(\lambda,\omega_\mu)=0$ by \eqref{eq_lim_lam_omegamu}.
In particular, we have $\calS_{0,+}\not=\emptyset$ and $d_+\leq 0$.

First we show the following result.
\begin{Lemma} \label{lem_ub_d+_leq} % \label{Lemma:7.2}
$\ub\leq d_+$.
\end{Lemma}

\begin{proof}
Since $\ub\leq 0$ due to \eqref{eq_general_est_b}, $\ub\leq d_+$ holds if $d_+=0$.
Therefore, we assume $d_+ < 0$ hereafter. 

To prove $\ub \leq d_+$, choose $u_0 \in \calS_{0,+}$ with $\calI(u_0) < 0$. 
In order to apply \cref{prop_paths_f}, we first notice that 
	\begin{equation*}
			{\frac{N-2}{2}}\norm{\nabla u_0}_2^2 -N\intRN G(u_0)\, dx
		= N\calI(u_0) -\norm{ \nabla u_0}_2^2 <0.
	\end{equation*}
Thus, there exists $\mu_0=e^{\lambda_0}>0$ such that
	\begin{equation} \label{eq_dim_ub_d+_pohoz} % \label{7.3}
		{\frac{N-2}{2}}\norm{\nabla u_0}_2^2 + N\left( {\frac{\mu_0}{2}}\norm{u_0}_2^2
		-\intRN G(u_0)\, dx\right) =0.
	\end{equation}
By setting $F(s)\coloneqq -{\frac{\mu_0}{2}}s^2+G(s)$, $f\coloneqq F'$ and defining $L$ by 
 \begin{equation*}
 L(u) \coloneqq \half\norm{\nabla u}_2^2 -\intRN F(u)\,dx:\, E\to \R ,
 \end{equation*}
clearly \eqref{eq_paths_g_1} holds.
Next we verify \eqref{eq_paths_g_2} when $N=2$.
In this case it follows from \eqref{eq_dim_ub_d+_pohoz} that $\mu_0\norm{u_0}_2^2=2\intRN G(u_0)\,dx$.
By $u_0\in\calS_{0,+}$, we have $Q(u_0)>0$, which implies 
	\begin{equation*}
			\begin{aligned} 
		\MoveEqLeft \intRN f(u_0)u_0\, dx =-\mu_0\norm{u_0}_2^2 +\intRN g(u_0)u_0\,dx 
		= \intRN g(u_0)u_0-2G(u_0)\,dx = Q(u_0) >0.	
		\end{aligned}
	\end{equation*}
Thus \eqref{eq_paths_g_2} holds. 
We apply \cref{prop_paths_f} to obtain a path 
$\gamma_0 : [0,1]\to E$ with properties 
	\begin{equation*}
		\max_{t\in [0,1]} L(\gamma_0(t)) \leq L(u_0)
		= \calI(u_0) %= \mathcal{I}(u_0)
		 + \mu_0 m_1,
		\quad 
		\gamma_0(0)=0, \quad L(\gamma_0(1))<0, \quad u_0 \in \gamma_0([0,1]).
	\end{equation*}
It is easily seen that $\gamma_0 \in \Lambda_{\mu_0}$ and $L=\Psi_{\mu_0}$, 
and \eqref{eq_amu_blambda} and \eqref{eq_b_blambda} yield 
$\ub\leq b(\lambda_0)\leq \calI(u_0)$. 
Since $u_0\in\calS_{0,+}$ with $\calI(u_0)<0$ is arbitrary, \cref{lem_ub_d+_leq} is proved. 
\end{proof}

When $N\geq 3$, we have $d=\ub$ by \eqref{eq_b_d_geq}, \eqref{eq_d+_d} and \cref{lem_ub_d+_leq}. 
We consider the case $N=2$ and assume $d<0$. 
Otherwise, since $d_+\leq 0$, the equalities $d=d_+$ and $d=\ub$ hold.
We study the behavior of $\calI$ on the set
	\begin{equation*}
			\calS_{0,-} \coloneqq \calS_0\setminus\calS_{0,+}
		=\big\{ u\in\calS_0 \mid Q(u)\leq 0\big\}.
	\end{equation*}
	Let us suppose $\calS_{0,-}\not=\emptyset$, and 
define
	\begin{equation*}
			d_- \coloneqq \inf_{u\in\calS_{0,-}}\calI(u).
	\end{equation*}
Clearly $d=\min\{d_+,d_-\}$ holds.
For $u_0\in\calS_{0,-}$ we consider the following curve on $\calS_0$:
	\begin{equation*}
			u_{0t}\coloneqq t^{1/2}u_0(t^{1/2}\cdot).
	\end{equation*}
It is easily seen that $\norm{u_{0t}}_2^2=\norm{u_0}_2^2$ for all $t>0$.
In particular, $u_{0t}\in\calS_0$ for all $t>0$. 

\begin{Lemma} \label{lem_d+_d} % \label{Lemma:7.3}
Assume $d<0$, $N=2$ and $\calS_{0,-}\not=\emptyset$. 
Then for all $u_0\in\calS_{0,-}$, there exists $t_0\in (0,1)$ such that
	\begin{equation} \label{eq_IQ_u0} % \label{7.4}
		\calI(u_{0t_0}) \leq \calI(u_0), \quad Q(u_{0t_0})>0.
	\end{equation}
As a consequence, $d_+=d$.
\end{Lemma}

\begin{proof}
For $u_0\in\calS_{0,-}$ and $t\in (0,1]$, we compute 
	\begin{equation} \label{eq_dt_Iu0} %\label{7.5}
	{\frac{d}{dt}} \calI(u_{0t}) = \half\norm{\nabla u_0}_2^2 
		-\half t^{-2} \int_{\R^2} %\intRN 
g(t^{1/2}u_0)t^{1/2}u_0-2G(t^{1/2}u_0)\, dx
		=\half t^{-1}\big( \norm{\nabla u_{0t}}_2^2 -Q(u_{0t})\big).	
	\end{equation}
Set
	\begin{equation*}
		T_0\coloneqq \inf \big\{ \ell\in (0,1] \mid Q(u_{0t})\leq 0 \ \text{for all}\ t\in [\ell,1] \big\}
			\in [0,1]. 	
	\end{equation*}
Since $u_0\in \calS_{0,-}$, we have $Q(u_0)\leq 0$ and $T_0\leq 1$.
To show the existence of the desired $t_0\in (0,1)$, it suffices to show
\begin{itemize}
\item[(a)] $T_0\in (0,1]$;
\item[(b)] $\calI(u_{0T_0}) \leq \calI(u_0)$;
\item[(c)] ${\frac{d}{dt}}\Big|_{t=T_0}\calI(u_{0t})\geq {\frac{1}{2T_0}}\norm{\nabla u_{0T_0}}_2^2>0$.
\end{itemize}

\noindent
In fact, by the definition of $T_0$ and (a), there exists $(t_j)_{j\in\N}\subset(0,T_0)$
such that $t_j\to T_0$ and $Q(t_j)>0$. Properties (b) and (c) imply $\calI(u_{0t_j})<\calI(u_0)$
for large $j$ and by setting $t_0=t_j$, property \eqref{eq_IQ_u0} holds.

To see (a), first we observe that $Q(u_{0t})>0$ for $t$ close to $0$. In fact, 
since $(g(s) s - 2G(s)) / s^4$ is bounded in $\R \setminus \{0\}$ and 
$(g(s) s - 2G(s)) /s^4 \to 1/2$ as $s \to 0$ thanks to \ref{(g1*)}, 
the dominated convergence theorem leads to 
	\begin{equation*}
			\begin{aligned}
	Q(u_{0t}) = t^{-1}\int_{\R^2} g(t^{1/2}u_0)t^{1/2}u_0-2G(t^{1/2}u_0)\, dx 
	&= t \int_{\R^2} {\frac{g(t^{1/2}u_0) t^{1/2}u_0 - 2G(t^{1/2}u_0) }{(t^{1/2}u_0)^4}}
			u_0^4\, dx \\
	&
	= t \left(\half\norm{u_0}_4^4+o(1)\right) \quad \text{as}\ t\sim 0. 
		\end{aligned}
	\end{equation*}
Thus for sufficiently small $t$, we have $Q(u_{0t})>0$ and (a) is verified.
Since ${\frac{d}{dt}}\calI(u_{0t}) > 0$ for $t\in [T_0,1]$ by \eqref{eq_dt_Iu0}, we have (b). 
Point (c) also follows from \eqref{eq_dt_Iu0}.

Thus for any $u_0\in \calS_{0,-}$ there exists $u_{0t}\in\calS_{0,+}$ with
$\calI(u_{0t})\leq \calI(u_0)$. Therefore we have $d_+\leq d_-$, which implies
$d_+=d$.
\end{proof}

\begin{proof}[Proof of \cref{thm_exists_min} (i)]
It follows from \cref{lem_ub_d+_leq} and \cref{lem_d+_d}.
\end{proof}

Finally, \cref{thm_exists_min} (ii) follows from \cref{thm_exists_min} (i) and \cref{thm_old_ex_precise}: 

\begin{proof}[Proof of \cref{thm_exists_min} (ii)]
By \cref{thm_exists_min} (i), we have $\ub=d < 0$. 
Hence, \cref{thm_old_ex_precise} yields the existence of a solution at level $\ub =d$. 
\end{proof}

\begin{proof}[Proof of \cref{cor_exists_min}]
\cref{prop_b_neg_post} (i) and \cref{thm_exists_min} (i) yield 
$\ob = d < 0$ and \cref{cor_exists_min} follows from \cref{thm_exists_min} (ii). 
\end{proof}

%%%%%%%%%%%%%%%%%%%%%%%%
\subsection{Properties of minimizers} 
This subsection is devoted to studying properties of $\calM_d$ where 
	\begin{equation*}
			\calM_d \coloneqq \big\{u \in \calS_0 \mid \calI(u) = d \big\}.
	\end{equation*}
First, we show the following relation.

\begin{Lemma} 
\label{lem_compact_projec} %\label{Lemma:7.5}
Assume \ref{(g0)}, \ref{(g1*)} and $d<0$. Then
	\begin{equation*}
			\calM_d = P_2(K_d),
	\end{equation*}
where $K_d$ is given in \eqref{eq_def_criticalset} % the set of critical points $(\lambda,u)$ of $I$ with $I(\lambda,u)=d$ 
and $P_2:\,\RE\to E$ is the standard projection.
\end{Lemma}

	\begin{proof}
The inclusion $P_2(K_d)\subset \calM_d$ is easily proved. On the other hand, let $u\in\calM_d$. Then there exists $\kappa \in \R$ such that 
	\begin{equation*}
			-\Delta u + \kappa u = g(u) \quad \text{in} \ \R^N. 
	\end{equation*}
Remark that $u$ satisfies the Pohozaev identity: 
	\begin{equation*}
			0 = \frac{N-2}{2} \| \nabla u \|_2^2 - N \left( \int_{\R^N} G(u) \, dx 
 - \frac{\kappa}{2} \| u \|_2^2 \right) 
		= N \calI(u) - \| \nabla u \|_2^2 + \frac{N\kappa }{2} \| u \|_2^2.
	\end{equation*}
Since $\calI(u) = d < 0$, we have $\kappa > 0$ and 
	$(\log \kappa, u) \in K_d$.
Thus $\calM_d\subset P_2(K_d)$. 
\end{proof}

Under \ref{(g0')} and \ref{(g1*)} with $\alpha = 0$, for $d<0$, $K_d$ is compact in $\RE$ thanks to \cref{prop_compact_Kb}. 
By virtue of \cref{lem_compact_projec} and \cref{thm_exists_min} (ii), we have the following statement.

\begin{Corollary} %\label{Corollary:7.6}
\label{cor_compact_Md}
Assume \ref{(g0')}, \ref{(g1*)} with $\alpha=0$ %{\rm \ref{(g1)}} 
and $d < 0$. 
Then the set $\calM_d$ is nonempty and compact in $E$. 
\end{Corollary}

\begin{Remark}\label{Rem:Md}
The outcome of \cref{cor_compact_Md} does not generally hold when $d=0$. Indeed, we can find examples such that one of the following facts occurs: 
\begin{itemize}
\item[{\rm (1)}] $\calM_{0}%d
$ is non-compact;
\item[{\rm (2)}] $\calM_{0}%d
=\emptyset$.
\end{itemize}
In fact, recalling $G_0(s) = \frac{1}{p+1} |s|^{p+1}$, 
we have the following facts.
\begin{itemize}
\item[{\rm (1)}] 
Consider the functional $\calI_0$. %$I_0(\lambda,u)=\half\norm{\nabla u}_2^2-\intRN G_0(u)$.
Then by \eqref{eq_values_b0} and \cref{thm_exists_min},
$d = \ub = 0$ and the corresponding set of minimizers is
	$	\calM_{0}%d
 = \{\pm \omega_\mu \mid \mu>0\}
	$.
In particular, $\calM_{0}%d
$ is non-compact in $E$. 
\item[{\rm (2)}] Assume that $g$ satisfies \ref{(g0')}, \ref{(g1*)} with $\alpha=0$, %\ref{(g1)}, 
\ref{(g3)} and
	\begin{equation*}
			G(s)<G_0(s) \quad \text{for all}\ s\not=0.
	\end{equation*}
Then, \cref{prop_b_neg_post} (ii) and \cref{thm_exists_min} (i) give $d=\ub=0$. Moreover for all $u\in\calS_0$
	\begin{equation*}
			\begin{aligned} 
		\calI(u) = \half\norm{\nabla u}_2^2 -\intRN G(u)\, dx 
		> \half\norm{\nabla u}_2^2 -\intRN G_0(u)\, dx 
			\geq 0. 		
		\end{aligned}
	\end{equation*}
Therefore $\calM_{0}%d
=\emptyset$ and $d=0$ is not a critical value of $I$. 
\end{itemize}
\end{Remark}

%%%%%%%%%%%%%%%%%%%%%%%%
\subsection{Nonexistence result: proof of Theorem \ref{thm_nonexis_min}} %\cref{thm_nonexis_min}}
\label{Section_nonexist_min} %\label{Subsection:9.2}

This subsection is devoted to the proof of \cref{thm_nonexis_min} and 
as stated at the beginning of \cref{Section_prelim}, 
we assume \ref{(g0)} instead of \ref{(g0')}. 
Under \ref{(g0)} and \ref{(rho+1)}, 
$H(s)=\half\rho(s)s^2$ satisfies 
\begin{equation}\label{H-ineq}
	\left\{\begin{aligned}
		&
		\text{$H(s) \leq \half\alpha s^2$ for all $s\in\R \setminus \{0\}$ and there exist $(s_n)_{n \in \N}$ and $(t_n)_{n \in \N}$ such that} 
		\\
		&
		\text{$s_n < 0 < t_n$, $s_n,t_n \to 0$, $H(s_n) < \frac{1}{2} \alpha s_n^2$ and 
		$H(t_n) < \frac{1}{2} \alpha t_n^2$}.
	\end{aligned}\right.
\end{equation}
As we will see, under these conditions we have $d=\inf_{u\in\calS_0}\calI(u)<0$, 
and thus it is natural to expect that $d$ is a critical value of $I$.
However we have the non-existence result stated in \cref{thm_nonexis_min}, that is,
$d$ is not attained in $\calS_0$.

\begin{proof} [Proof of \cref{thm_nonexis_min}]
First we show $d=-\alpha m_1<0$. 
To get $d \leq - \alpha m_1$, we use $\omega_\mu %(\cdot) 
= \mu^{N/4} \omega_1 (\sqrt{\mu} \cdot)$ 
in \eqref{eq_omegamu}. The fact $\calI_0 (\omega_\mu) = 0$, \ref{(rho+1)} and the dominated convergence theorem imply 
that as $\mu \to +\infty$, 
\[
\begin{aligned}
	d \leq \calI(\omega_\mu) 
	= \calI_0(\omega_\mu) - \int_{\RN} H(\omega_\mu) \, dx 
	= - \frac{1}{2} \int_{\RN} \rho \big( \mu^{N/4} \omega_1 \big) \omega_1^2 \, dx
	\to - \alpha m_1. 
\end{aligned}
\]
Hence, $d \leq - \alpha m_1$ holds. 

	On the other hand, 
from \ref{(rho+1)} it follows that 
\begin{equation*}
	\calI_0(u) -{\frac{\alpha}{2}}\norm u_2^2 \leq \calI_0 (u) - \int_{\RN} H(u) \, dx = \calI(u) 
	\qquad \text{for all}\ u\in\calS_0;
\end{equation*}
being $\calI_0(u)\geq 0$ by \eqref{eq_estimate_GN} we obtain $d\geq -\alpha m_1$ and thus $d = - \alpha m_1$. 

Next we show that $d$ is not attained. Arguing indirectly, we suppose that
$u\in\calS_0$ attains $d$. Then we have
	\begin{equation*}
			\calI_0(u) -\intRN H(u)\, dx = d = -\alpha m_1.
	\end{equation*}
Since $\calI_0(u)\geq 0$, the fact $\| u \|_2^2 = 2m_1$ yields 
	\begin{equation*}
			\intRN H(u)-{\frac{\alpha}{2}}u^2 \, dx \geq 0.
	\end{equation*}
However, by \eqref{H-ineq}, we have a contradiction $\int_{\RN} H(u) \, dx <  \alpha \| u \|^2_2 / 2 $ 
since $u \in E \subset C(\RN \setminus \{0\})$ and $|u|(\RN) \cap (0,s) \neq \emptyset$ for every $s > 0$ due to $u \not \equiv 0$ (cf. \cite[Lemma 2.4]{BMS25}). 
\end{proof}

%%%%%%%%%%%%%%%%%%%%%%%%%%%%%%%%%%%%%%%%%%%%%%%%%%%%%%%%%%%%%%%%%%
\section{Stability properties and comments on \ref{(g2)}}
\label{sec_stability}

This section is devoted to the proof of \cref{Thm:Exw/og2} and \cref{Cor:ex-g}. 
The proof is based on the continuity property for non-existence
of solutions to the zero-mass problem \eqref{eq_zero_mass}. 
To motivate \cref{Prop:non-ex} below, we note from the proof of \cite[Theorem 1.1]{CGIT24} that 
in order to show that $\ob_G > 0$ is a critical value of $I$, it is sufficient to assume -- rather than \ref{(g2)} -- the non-existence of positive solutions $u$ of \eqref{eq_zero_mass} with $Z(u) = \ob_G > 0$, where $Z$ is a functional corresponding to \eqref{eq_zero_mass} and defined by 
	\[
	Z(u) \coloneq \intRN \frac{1}{2} \abs{\nabla u}^2 - G(u) \, dx;
	\]
on the other hand, it is not clear if this condition is stable under perturbations of $g$.
Therefore we consider the following condition \ref{(g2w)}, which is %sufficient to show the criticality of $\ob_G$ but is 
weaker than \ref{(g2)}: for a fixed $\beta \geq %>
0$ we define
\begin{enumerate}[label={\rm (g\arabic*$\beta$)}]
	\setcounter{enumi}{1}
	\item \label{(g2w)}
	there are no solutions $u$ of \eqref{eq_zero_mass} with $u \in F \setminus \{0\} $, $ u \geq 0$ and $Z%Z_G
	(u) \leq \beta$.
\end{enumerate}
Condition \ref{(g2w)} for $\beta \geq \ob_G>0$ is sufficient to show the criticality of $\ob_G$.
In \cref{section_stab_nonext_zero}, we show that \ref{(g2w)} is stable under perturbations of $g$, 
and in \cref{section_exist_perturb}, \cref{Thm:Exw/og2} and \cref{Cor:ex-g} are proved. 

%%%%%%%%%%%%%%%%%%%%%%%%
\subsection{Continuity property of non-existence of solutions to the zero mass problem} 
\label{section_stab_nonext_zero}

Throughout this subsection, we consider the non-existence of solutions to \eqref{eq_zero_mass} in a slightly general case. 
Let $q \in (1,2^*-1)$ and introduce the following function space and norm: %\orange{M: here ``$G$ is even'' is removed.}
\[
\begin{aligned}
	\calG_q &\coloneqq \big\{ G \in C^1(\R) \mid G(0) = 0, \ \text{$G$ is even}, 
	\ \norm{G}_{\calG_{q}} < \infty \big\},
	\quad 
	\norm{G}_{\calG_q} \coloneqq \sup_{s \in \R \setminus \{0\} } 
	\frac{ \abs{G(s)} + \abs{G'(s)s} }{\abs{s}^{q+1}}.
\end{aligned}
\]
We also introduce the following generalization %weakening 
of \ref{(g1*)}, % \red{with $\alpha=0$}, 
that is
\begin{enumerate}[label={\rm (g1*$q$)}] %{\rm (g1*q)}
\item \label{(g1*q)} $g$ satisfies %$G \in \calG_q$ and 
 \begin{equation*}
 \lim_{s \to 0,\pm \infty}{\frac{g(s)%s
 }{|s|^{q-1}s}} =1.
 \end{equation*}
\end{enumerate}
We observe that, if \ref{(g0')} %\ref{(g0)} %\ref{(g0')} 
and \ref{(g1*q)} hold, then $G \in \calG_q$ and %moreover
\[
\lim_{s \to 0, \pm \infty} \frac{G(s)}{|s|^{q+1} } = \frac{1}{q+1};
\]
moreover, \ref{(g1*)} implies \ref{(g1*q)} with $q=p$. %, while we observe that if \ref{(g0')} and \ref{(g1*q)} hold, then $G \in \calG_q$.} 
For any $G \in \calG_{q}$, we consider the zero mass problem \eqref{eq_zero_mass} and use the following functional: 
\begin{equation*}
	Z_G(u) \coloneqq \half\norm{\nabla u}_2^2-\intRN G(u)\, dx : F_q \to \R,
\end{equation*}
where 
\[
F_q \coloneqq \big\{ u \in L^{q+1} (\RN) \mid \nabla u \in L^2(\RN), u(x) = u(\abs{x})\big\}, 
\quad \norm{u}_{F_q} \coloneqq \norm{\nabla u}_2 + \norm{u}_{q+1}. 
\]
We first collect some properties of $F_q$. 

\begin{Lemma}\label{Lem_basic-Fq}
The following facts hold.
	\begin{enumerate}[label={\rm (\roman*)}]
		\item 
		The space $C^\infty_{0,r}(\RN) \coloneqq \big\{ u \in C^\infty_0 (\RN) \mid u(x) = u(|x|) \big\}$ is dense in $F_q$. 
		\item 
		There exists $C=C(q,N)>0$ such that for any $R>0$ and $u \in F_q$, 
		\begin{equation}\label{eq_decay-Fq}
			\abs{u(R)} \leq C R^{ - 2(N-1)/(q+2) } \norm{ u }_{L^q(\RN \setminus B_R(0))}^{q/ (q+2) } \norm{\nabla u}_{L^2(\RN \setminus B_R(0))}^{2/(q+2)}.
		\end{equation}
		Furthermore, $F_q \subset C(\RN \setminus \{0\})$ holds. 
		\item 
		For any $r \in [q+1, 2^\ast]$ when $N \geq 3$ and $r \in [q+1,\infty)$ when $N = 2$, 
		the embedding $F_q \subset L^r(\RN)$ holds. 
		Furthermore, the embedding $F_q \subset L^r(\RN)$ is compact for every $r \in (q+1,2^*)$. 
	\end{enumerate}	
\end{Lemma}

\begin{proof}
Assertion (i) is easy to prove. 
For assertion (ii), the inclusion $F_q \subset C(\RN \setminus \{0\})$ is deduced from 
the fact that by (i) and \eqref{eq_decay-Fq}, 
each function $u \in F_q$ can be approximated by functions in $C^\infty_{0,r}(\RN)$ in $L^\infty_{\rm loc} ((0,\infty))$. 
Thus, it suffices to show \eqref{eq_decay-Fq} for $u \in C^\infty_{0,r}(\RN)$ by virtue of (i). 
Let $u \in C^\infty_{0,r} (\RN)$ and set $\gamma \coloneqq (q+2)/2 \geq 1$. 
Then 
\[
\begin{aligned}
	\abs{u(R)}^\gamma 
	&\leq \int_R^\infty \gamma \abs{u(r)}^{\gamma -1} \abs{u'(r)} \, dr
	\\
	&\leq \gamma \left( \int_R^\infty r^{-(N-1)} \abs{u(r)}^{2\gamma - 2} \,dr \right)^{1/2} 
	\left( \int_R^\infty r^{N-1} \abs{u'(r)}^2 \, dr \right)^{1/2}
	\\
	&\leq 
	C_{q,N} \left( \int_R^\infty r^{ -2(N-1) } r^{N-1} \abs{u(r)}^q \, dr \right)^{1/2} \norm{\nabla u}_{L^2(\RN \setminus B_R(0))} 
	\\
	&\leq 
	C_{q,N} R^{ -(N-1) } \norm{u}_{L^q(\RN \setminus B_R(0))}^{ q/2 } \norm{\nabla u}_{L^2(\RN \setminus B_R(0))} .
\end{aligned}
\]
Hence, 
\[
\abs{u(R)} \leq C_{q,N} R^{ -2(N-1)/(q+2) } 
\norm{u}_{L^q(\RN \setminus B_R(0))}^{ q/ (q+2) } \norm{\nabla u}_{L^2(\RN \setminus B_R(0))}^{ 2/(q+2) } 
\]
and \eqref{eq_decay-Fq} holds. 
Finally, assertion (iii) may be proved as in \cite[Lemma 2.3 (ii)]{CGIT24} and we omit the details. 
\end{proof}

It is easily seen that $Z_G \in C^1(F_q,\R)$ for any $G \in \calG_{q}$ and 
critical points of $Z_G$ are solutions of \eqref{eq_zero_mass}. 
For any $\beta \in \R$, $\calK_{q}(G,(-\infty,\beta])$ denotes the set of critical
points of $Z_G$ whose critical values are in $(-\infty,\beta]$: 
\begin{equation*}
	\calK_{q}(G, (-\infty,\beta] ) 
	\coloneqq 
	\big\{ u\in F_q \mid 
		Z_G'(u)=0, \ Z_G(u) \leq \beta\big\}; \quad 
\calK_q^+(G, (-\infty,\beta] ) \coloneqq \calK_q(G, (-\infty, \beta]) \cap \{u\geq 0\}.
\end{equation*}
Notice that, in particular, \ref{(g2w)} means $\calK_{p}^+(G, (-\infty,\beta] ) =\{0\}$, and \ref{(g2)} means $\calK_{p}^+(G, (-\infty,+\infty) ) =\{0\}$.
For \eqref{eq_zero_mass} with $G \in \calG_{q}$ and $g = G'$, 
the argument in \cite[Proposition 1]{BeLi83I} is applicable and 
we may show that each critical point $u$ of $Z_G$ satisfies the Pohozaev identity: 
\begin{equation}\label{Pohoid}
	NZ_G(u) - \norm{\nabla u}_2^2 = \frac{N-2}{2} \norm{\nabla u}_2^2 - N \int_{\RN} G(u) \, dx = 0. 
\end{equation}
In particular, $Z_G(u) \geq 0$ holds for each $u \in F_q$ with $Z_G'(u) = 0$. 

\begin{Remark}
In \cite[Proposition 1.2 (i)]{CGIT24} we notice that, thanks to the results in \cite{AGQ16, AS11}, condition \ref{(g2)} holds when $N=2,3,4$, since for these values we have $p\leq\frac{N}{N-2}$. Arguing in a similar way, if $q \in (1, \frac{N}{N-2}]$ and $g$ satisfies \ref{(g0')} 
and \ref{(g1*q)}, %\in \calG_{q,0}$, 
then we can show more generally that $\calK_q^+(G, (-\infty, +\infty)) = \{0\}.$
\end{Remark}

The zero-mass problem \eqref{eq_zero_mass} has the following continuity property, which has an interest on its own.

\begin{Proposition} \label{Prop:non-ex}
	Let $\beta \in [0,\infty)$, $q \in (1,2^*-1)$, $g$ satisfy \ref{(g0')} 
	 and \ref{(g1*q)}, %\in \calG_{q,0}$, 
	 and $\calK_{q}( G, (-\infty,\beta] ) = \{0\} $. 
	Then there exists $\rho_{\beta}>0$ small such that
	$\calK_{q}(G+\Xi, (-\infty , \beta])=\{0\}$ holds for all $\Xi \in\calG_{q}$ with
	$\norm{\Xi}_{\calG_{q}}<\rho_{\beta}$. 
	A similar statement holds also for $\calK^+_q$.
	In particular, when $g_1$ fulfills \ref{(g0')}, \ref{(g1*)} with $\alpha = 0$ and \ref{(g2w)} for some $\beta \in [0,\infty)$, 
	then there exists $\rho_{\beta}>0$ such that \ref{(g2w)} also holds for $g = g_1 + \Xi'$ with $\Xi \in \calG_p$, $\norm{\Xi}_{\calG_{p}} < \rho_{\beta}$. 
\end{Proposition}

\begin{proof}
We argue indirectly and assume the existence of a sequence 
$(\Xi_j)_{j\in\N}\subset\calG_{q}$ such that
\begin{equation*}\label{eq:9.1}
	\norm{\Xi_j}_{\calG_{q}} \to 0, \quad 
	\calK_q\big(G+\Xi_j, (-\infty,\beta]\big)\not=\{0\}.
\end{equation*}
Let thus $u_j\in \calK_q\big(G+\Xi_j, (-\infty,\beta]\big)\setminus\{0\}$ for $j\in\N$ and assume 
	\begin{equation} \label{eq:9.2}
		\norm{\Xi_j}_{\calG_q}< \frac{1}{4} \quad \text{for all}\ j.
	\end{equation}
We write $g\coloneqq G'$, %$g(s)\coloneqq G'(s)$, 
$h(s) \coloneqq g(s) - \abs{s}^{q-1} s$ and $\xi_j \coloneqq \Xi_j'$. %$\xi_j(s)\coloneqq \Xi_j'(s)$.

\smallskip

\noindent
\textbf{Step 1:} \textsl{$(u_j)_{j\in\N}$ is bounded in $F_q$.}

\smallskip

The facts $u_j \in \calK_q \big( G+\Xi_j, (-\infty , \beta] \big) $ and \eqref{Pohoid} give 
	\begin{align} 	
		&\half\norm{\nabla u_j}_2^2 -\intRN G(u_j)+\Xi_j(u_j)\, dx \leq \beta,	\label{eq:9.3}\\
		&\frac{N-2}{2}\norm{\nabla u_j}_2^2 -N\intRN G(u_j)+\Xi_j(u_j)\, dx=0.
			\label{eq:9.4}
	\end{align}
It follows from \eqref{eq:9.3} and \eqref{eq:9.4} that 
$\norm{\nabla u_j}_2^2 \leq N \beta$ and $\norm{\nabla u_j}_2^2$ is bounded.

We next argue as in \cite[Proposition 5.1]{HIT10} to prove that $\norm{u_j}_{q+1}$ is bounded. To this end, 
by \ref{(g1*q)},
the following hold: 
\begin{equation}\label{h*-super}
\lim_{s \to 0} \frac{h(s)}{ \abs{s}^{q-1}s } = 0 = \lim_{\abs{s} \to \infty} \frac{h(s)}{ \abs{s}^{q-1}s }. 
\end{equation}
When $N \geq 3$, by \eqref{eq:9.2}, 
there exists $C_0 > 0$ such that for all $j \geq 1$ and $s \in \R$, 
\[
\abs{h(s)} + \abs{\xi_j(s)} \leq \frac{1}{2} \abs{s}^q + C_0 \abs{s}^{2^*-1}. 
\]
Hence, the fact $Z_{G+\Xi_j} '(u_j) u_j = 0$ leads to 
\[
\norm{\nabla u_j}_2^2 = \int_{\RN} \abs{u_j}^{q+1} + h(u_j) u_j + \xi_j(u_j) u_j \, dx 
\geq \frac{1}{2} \norm{u_j}_{q+1}^{q+1} - C_0 \norm{u_j}_{2^*}^{2^*} \geq \frac{1}{2} \norm{u_j}_{q+1}^{q+1} - C \norm{\nabla u_j}_2^{2^*}.
\]
The boundedness of $\norm{\nabla u_j}_2$ implies that $\norm{u_j}_{q+1}$ is bounded. 

When $N=2$, we argue indirectly and suppose $\norm{u_j}_{q+1} \to \infty$. Set 
\[
\tau_j \coloneq \norm{u_j}_{q+1}^{-(q+1)/2}, %\frac{1}{\norm{u_j}_{q+1}^{(q+1)/2} }, 
\quad w_j \coloneq u_j(\cdot/\tau_j). % \left( \frac{\cdot}{\tau_j} \right).
\]
It is immediate to see $\norm{\nabla w_j}_2^2 = \norm{\nabla u_j}_2^2$ and $\norm{w_j}_{q+1} = 1$. 
Let $w_j \rightharpoonup w_\infty$ weakly in $F_q$. Notice that 
\[
-\tau_j^2\Delta w_j = \abs{w_j}^{q-1}w_j + h(w_j) + \xi_j (w_j) \quad \text{in} \ \R^2.
\]
Since 
\begin{equation}\label{eq_bound_xijq}
\abs{\xi_j(s)} \leq \norm{\Xi_j}_{\calG_q} \abs{s}^{q} \quad \text{for each $s \in \R$}, \quad \norm{\Xi_j}_{\calG_q} \to 0,
\end{equation}
it follows from \cref{Lem_basic-Fq} that for every $\varphi \in C^\infty_{0,r}(\R^2)$, 
\[
\begin{aligned}
	0 = \lim_{j \to \infty} \tau_j^2 \int_{\R^2} \nabla w_j \cdot \nabla \varphi \, dx 
	&= \lim_{j \to \infty} \int_{\R^2} \left( |w_j|^{q-1}w_j + h(w_j) + \xi_j(w_j) \right) \varphi \, dx 
	\\
	&= \int_{\R^2} \left( \abs{w_\infty}^{q-1} w_\infty + h(w_\infty) \right) \varphi \, dx. 
\end{aligned}
\]
In particular, $\abs{w_\infty}^{q-1} w_\infty + h(w_\infty) = 0$ holds. 
From $F_q \subset C(\R^2 \setminus \{0\})$, \eqref{h*-super} and $w_\infty(x) \to 0$ as $\abs{x} \to \infty$ thanks to \cref{Lem_basic-Fq}, 
we deduce $w_\infty = 0$. 

Since \cref{Lem_basic-Fq}, Strauss' lemma (see \cite[Compactness Lemma 2]{Str77} or \cite[Theorem A.I]{BeLi83I}), 
\eqref{h*-super} and $\norm{\Xi_j}_{\calG_q} \to 0$ lead to 
\[
\lim_{j \to \infty}\int_{\R^2} h(w_j)w_j \, dx = 0, 
\quad \limsup_{j \to \infty} \int_{\R^2} \abs{\xi_j(w_j) w_j} dx 
\leq \limsup_{j \to \infty} \norm{\Xi_j}_{\calG_q} \norm{w_j}_{q+1}^{q+1} = 0, 
\]
and the following contradiction is obtained: 
\[
1 = \norm{w_j}_{q+1}^{q+1} = \lim_{j \to \infty} \int_{\R^2} \abs{w_j}^{q+1} + h(w_j)w_j + \xi_j(w_j) w_j \, dx 
= \lim_{j \to \infty} \tau_j^2 \norm{\nabla w_j}_2^2 = 0.
\]
Hence $\norm{u_j}_{q+1}$ is bounded when $N = 2$. 

\smallskip

By Step 1, we may assume that $u_j \wlimit u_0$ weakly in $F_q$ for some $u_0 \in F_q$. 
From $\abs{\Xi_j(s)} \leq \norm{\Xi_j}_{\calG_q} \abs{s}^{q+1}$ and $\abs{\xi_j(s)} \leq \norm{\Xi_j}_{\calG_q} \abs{s}^q$, 
we may verify 
\begin{equation}\label{eq:H_jh_j}
	\int_{\RN} \abs{\Xi_j(u_j)} \, dx \to 0, \quad \norm{\xi_j(u_j)}_{(q+1)/q} \to 0. 
\end{equation}

\smallskip

\noindent
\textbf{Step 2:} \textsl{$u_j\to u_0$ strongly in $F_q$.}

\smallskip

Notice first that $Z_{G}'(u_0) = 0$ due to $\norm{\Xi_j}_{\calG_q} \to 0$, $Z_{G+\Xi_j}'(u_j) = 0$ and 
$u_j \rightharpoonup u_0$ weakly in $F_q$. 
We set $v_j\coloneqq u_j-u_0$. From \eqref{h*-super} and $\sup_{j} \abs{u_j(x)} \to 0$ as $\abs{x} \to \infty$ by \cref{Lem_basic-Fq}, it follows that 
\begin{equation}\label{h*u_jv_j}
	\begin{aligned}
		&\int_{\RN} h(u_j) u_j dx = \int_{\RN} h(u_0)u_0 \, dx + o(1), \quad 
		\int_{\RN} H(u_j) \, dx = \int_{\RN} H(u_0) \, dx + o(1),
		\\
		& \int_{\RN} h(v_j)v_j dx = o (1) = \int_{\RN} H(v_j) \, dx. 
	\end{aligned}
\end{equation}
Hence, the Brezis-Lieb lemma leads to 
\begin{align} 	
	&\intRN G(u_j) \,dx = \intRN G(u_0)\, dx +\intRN G(v_j)\, dx +o(1), \label{eq:9.5}\\
	&\intRN g(u_j)u_j \,dx = \intRN g(u_0)u_0\, dx +\intRN g(v_j)v_j\, dx +o(1).
	\label{eq:9.6}
\end{align}
Furthermore, \eqref{eq:9.4} and $Z_{G+\Xi_j}'(u_j) u_j = 0 $ yield 
\begin{equation}\label{eq:PI-GXi_j}
	\frac{N-2}{2}\intRN g(u_j)u_j +\xi_j(u_j)u_j\, dx
	-N \intRN G(u_j)+\Xi_j(u_j)\, dx = 0
\end{equation}
and \eqref{eq:H_jh_j} gives 
\begin{equation} \label{eq:9.7}
	\frac{N-2}{2}\intRN g(u_j)u_j\, dx -N \intRN G(u_j)\, dx \to 0.
\end{equation}
Since $Z'_{G} (u_0) = 0$, we also have 
\begin{equation} \label{eq:9.8}
	\frac{N-2}{2}\intRN g(u_0)u_0\, dx -N \intRN G(u_0)\, dx = 0.
\end{equation}
Thus, by \eqref{h*u_jv_j}--\eqref{eq:9.8}, 
\[
\begin{aligned}
	o(1) 
	&= \int_{\RN} G(u_j) - \frac{N-2}{2N}g(u_j) u_j \, dx 
	\\
	&= \int_{\RN} G(v_j) - \frac{N-2}{2N} g(v_j) v_j \, dx + o(1) = \left(\frac{1}{q+1} - \frac{N-2}{2N}\right) \norm{v_j}_{q+1}^{q+1} + o(1),
\end{aligned}
\]
which implies $\norm{v_j}_{q+1} \to 0$ and $u_j \to u_0$ strongly in $L^{q+1} (\RN)$. Finally, from 
\[
\norm{\nabla u_j}_2^2 = \int_{\RN} g(u_j) u_j + \xi_j(u_j) u_j \, dx \to \int_{\RN} g(u_0) u_0 \, dx = \norm{\nabla u_0}_2^2, 
\]
we infer that $\norm{\nabla u_j - \nabla u_0}_2 \to 0$ and $u_j\to u_0$ strongly in $F_q$.

\smallskip

\noindent
\textbf{Step 3:} \textsl{Conclusion.}

\smallskip

\noindent
By Step 2, we have $Z_{G}'(u_0) = 0$ and 
\begin{equation*}
	\begin{aligned}
		Z_{G}(u_0)&= \lim_{j\to\infty} Z_{G}(u_j) 
		= \lim_{j\to\infty} \left(Z_{G+\Xi_j}(u_j) + \intRN \Xi_j(u_j)\, dx\right) \leq \beta. 
	\end{aligned}
\end{equation*}
To complete the proof, it suffices to show $u_0 \not \equiv 0$ since this yields $u_0\in\calK_q(G, (-\infty,\beta])\setminus\{0\}$, 
which contradicts the assumption $\calK_q(G,(-\infty ,\beta])=\{0\}$. 
Suppose on the contrary that $u_0 \equiv 0$. 
Since $\norm{u_j - u_0}_{F_q} \to 0$ and $u_j$ is a solution to 
\[
-\Delta u_j = g (u_j) + \xi_j(u_j) \quad \text{in} \ \RN, 
\]
the bound \eqref{eq_bound_xijq} %facts $|\xi_j(u_j)| \leq \norm{\Xi_j}_{\calG_q} |u_j|^{q}$ and $\norm{\Xi_j}_{\calG_q} \to 0$ 
with elliptic regularity 
implies $\norm{u_j}_\infty = \norm{u_j - u_0}_\infty \to 0$. 
From \eqref{eq:PI-GXi_j}, \eqref{eq_bound_xijq}, \ref{(g1*q)}, %\red{} %$G \in \calG_{q,0}$, 
$q \in (1,2^*-1)$ and $u_j \not \equiv 0$, it follows that for sufficiently large $j$, 
\[
\begin{aligned}
	0 &= \int_{\RN} \frac{N-2}{2N} \big( g(u_j) u_j + \xi_j(u_j) u_j \big) - G(u_j) - \Xi_j(u_j) \, dx 
	\\
	&= \int_{\RN} |u_j|^{q+1} \left( \frac{N-2}{2N} - \frac{1}{q+1} - o(1) \right) dx < 0.
\end{aligned}
\]
This is a contradiction and $u_0 \not \equiv 0$ holds. The case $\calK^+_q$ follows in a similar way, once we notice that $u_j \geq 0$ implies $u_0 \geq 0$.
\end{proof}

%%%%%%%%%%%%%%%%%%%%%%%%
\subsection{Proof of Theorem \ref{Thm:Exw/og2} and Corollary \ref{Cor:ex-g}} % \cref{Thm:Exw/og2} and \cref{Cor:ex-g}} 
\label{section_exist_perturb}

As an application of \cref{Prop:non-ex}, 
we show %prove 
\cref{Thm:Exw/og2} and \cref{Cor:ex-g}. 
Let $g_1$ satisfy \ref{(g0')} and \ref{(g1*)} with $\alpha = 0$. 
Before proving the results, %them, 
we recall the space $X$ from \eqref{def:X}-\eqref{def:normX} and notice that 
\begin{enumerate}[label=(\roman*)]
	\item 
	$g_1 + \Xi$ satisfies \ref{(g0')} and \ref{(g1*)} with $\alpha = 0$ (thus, in particular, \ref{(g1*q)} with $q=p$) %\subset \calG_{p,0}$ \orange{K: remove $\calG_{p,0}$} 
	for 
	any $\Xi \in X$;
	\item 
	for any $\Xi \in X$, $\norm{\Xi}_{\calG_p} \leq 2\norm{\Xi}_{X}$ holds. % where $\norm{\cdot}_X$ is defined in \eqref{def:normX}. 
\end{enumerate}
%\red{In this Subsection we fix $g_1$ satisfying  \ref{(g0')} and \ref{(g1*)} with $\alpha=0$.} \blue{M: \underline{Is it right?}}
For any $\Xi \in X$, we define $I_{G_1+\Xi} \in C^1(\R \times E, \R)$ and $\Psi_{\mu,G_1+\Xi} \in C^1(E,\R)$ by (recall $\mu=e^{\lambda}$)
\[
I_{G_1+\Xi}(\lambda,u) \coloneqq \half\norm{\nabla u}_2^2-\intRN G_1(u)+\Xi(u)\,dx
+\mu\left( \half\norm u_2^2-m_1\right) 
\eqqcolon \Psi_{\mu,G_1+\Xi}(u) - \mu m_1. 
\]
We also consider balls in $X$ by writing, for any $\rho>0$, 
\[B^X_{\rho}\coloneqq \big\{ \Xi \in X \mid \norm{\Xi}_X <\rho \big\}.\]
Write also $h_1(s) \coloneqq g_1(s) - \abs{s}^{p-1} s$ and $\xi \coloneqq \Xi'$. 

To prove \cref{Thm:Exw/og2}, we need a uniform $2$-dimensional path class for $I_{G_1+\Xi}$ with respect to $\Xi$. 
For this purpose, by \eqref{def:X} and \eqref{def:normX}, there exists $A_1 \geq 1$ such that 
\begin{equation}\label{eq_unif_H1Xi}
	\abs{H_1(s)} + \abs{\Xi(s)} \leq A_1 s^2 \quad \text{for all $s \in \R$ and $\Xi \in B^X_{1/3}$}. 
\end{equation}
Notice that \eqref{eq_estimate_GN} and \eqref{eq_unif_H1Xi} imply that 
for every $(\lambda,u) \in \R \times E$ and $\Xi \in B^X_{1/3}$ with $\| u \|_2^2 \leq 2m_1$, 

\begin{equation}\label{eq_bddest1}
	I_{G_1+\Xi} (\lambda, u) = \frac{1}{2} \norm{\nabla u}_2^2 - \frac{1}{p+1} \norm{u}_{p+1}^{p+1} - \int_{\RN} H_1(u) + \Xi(u) \,dx 
	\geq - 2 A_1 m_1.
\end{equation}
We next prove the existence of $\rho_0 \in (0,1/3)$ and $\zeta_1 \in C^2(\R,E)$, which are independent of $\Xi$, such that 
\eqref{eq_zeta0_prop} holds with $I_{G_1+\Xi}$ and $\Psi_{\mu,G_1+\Xi}$ for any $\Xi \in B^X_{\rho_0}$.

\begin{Lemma}\label{Lem_zeta1rho1}
	There exist $\rho_0 \in (0,1/3)$ and $\zeta_1 \in C^2( \R, E )$ such that $\zeta_1\geq 0$ and
	\begin{enumerate}[label={\rm (\roman*)}]
		\item 
		for any $\lambda \in \R$, 
		$\norm{\zeta_1(\lambda)}_2^2 > 2m_1$; 
		\item 
		for each $\lambda \in \R$ and $\Xi \in B^X_{\rho_0}$, 
		$I_{G_1+\Xi} (\lambda, \zeta_1 (\lambda) ) \leq - 2 A_1 m_1 - 1 - \mu m_1 $; 		
		\item 
		for every $\Xi \in B^X_{\rho_0}$, $\max_{0 \leq t \leq 1} I_{G_1+\Xi} (\lambda, t\zeta_1(\lambda) ) \to 0$ as $\abs{\lambda} \to \infty$; 
		here the convergence may depend on $\Xi \in B^X_{\rho_0} $. 
	\end{enumerate}
\end{Lemma}

\begin{proof}
The following argument is a refinement of \cite[Lemma 3.3 and Subsection 3.3%3.4
]{CGIT24}. 
By \eqref{def:normX}, there exists $s_0 > 0$ such that
\begin{equation}\label{eq_est-lsmall}
\abs{H_1(s)} + \abs{\Xi(s)} \leq \frac{1}{2} %\cdot 
\frac{1}{p+1} \abs{s}^{p+1} \quad \text{for each $s \in [0,s_0]$ and $\Xi \in B^X_{1/3}$}. 
\end{equation}
Arguing as in \cite[Proof of Lemma 3.3 (i)]{CGIT24} and fixing $T_0 > 1$ so that 
\[
- \frac{T_0^{p+1}}{2(p+1)} \norm{\omega_1}_{p+1}^{p+1} < -4 A_1 m_1 - 2, 
\]
we may prove via \eqref{eq_est-lsmall} that 
\[
\limsup_{\lambda \to - \infty } \sup_{\Xi \in B^X_{1/3}} I_{G_1+\Xi} (\lambda, T_0 \mu^{-1/(p+1)} \omega_\mu ) < -4 A_1 m_1 - 2.
\]
Hence there exists $\lambda_0 \in \R$, $\mu_0 \coloneqq e^{\lambda_0}$, such that for all $\lambda \in (-\infty,\lambda_0]$ 
%\blue{M: I have removed "and $\Xi \in B^X_{1/3}$" and kept the supremum, is it okay? I moreover added the second line, \underline{is it okay?}}, 
\begin{equation}\label{eq_est_ab1}
	\sup_{\Xi \in B^X_{1/3}} I_{G_1+\Xi} ( \lambda, T_0 \mu^{-1/(p+1)} \omega_\mu ) < -2 A_1 m_1 - 2 -\mu m_1.
\end{equation}
On the other hand, it follows from \eqref{eq_unif_H1Xi} that for all $t \in \R$, % and 
$\Xi \in B^X_{1/3}$ and $\lambda \in \R$, 
\[
I_{G_1+\Xi}(\lambda, t\omega_\mu) \leq \mu \left( \Psi_{01} (t \omega_1) + \frac{2t^2 A_1 m_1}{\mu} \right) - \mu m_1.
\]
Therefore, we may find $T_1 \geq T_0$ such that 
\begin{equation}\label{eq_est_ab2}
	\sup_{\Xi \in B^X_{1/3}} I_{G_1+\Xi} (\lambda, T_1 \omega_\mu) < -2A_1 m_1 - 2 - \mu m_1\quad \text{for all $\lambda \geq \lambda_0$}. 
\end{equation}
	
To construct $\zeta_1\in C^2(\R,E)$, notice that, by \eqref{eq_est_ab1} and \eqref{eq_est_ab2} with $\mu=\mu_0 %\coloneqq e^{\lambda_0}
$ and $\Xi=0$, we have $\Psi_{\mu_0, G_1} (T_0 \mu_0^{-1/(p+1)} \omega_{\mu_0} ) < -2A_1 m_1 - 2$ 
and $\Psi_{\mu_0, G_1} (T_1 \omega_{\mu_0}) < -2A_1 m_1 -2$. % where $\mu_0 \coloneqq e^{\lambda_0}$. 
Since the set $\{u \in E \mid \Psi_{\mu_0, G_1} < -2A_1 m_1 - 2 \}$ is path-connected (see \cite[Proposition 3.1 (iii)]{CGIT24}), 
there exists $\sigma_1 \in C( [0,1] , E )$ such that 
\[
\sigma_1(0) = T_0 \mu_0^{-1/(p+1)} \omega_{\mu_0}, \quad \sigma_1(1) = T_1 \omega_{\mu_0}, \quad 
\Psi_{\mu_0, G_1} (\sigma_1(t)) < -2A_1 m_1 - 2 \quad \text{for all $t \in [0,1]$}. 
\]
By changing $\sigma_1(t)$ to $\abs{\sigma_1(t)}$ and noting $\Psi_{\mu_0, G_1} (\sigma_1(t)) = \Psi_{\mu_0, G_1} (\abs{\sigma_1(t)})$, 
without loss of generality we may suppose $\sigma_1(t) \geq 0$ for each $t \in [0,1]$. 
Furthermore, $\norm{\sigma_1 (t)}_2^2 > 2m_1$ holds due to 
$I_{G_1} (\lambda_0, \sigma_1(t)) < - 2A_1 m_1 - 2- \mu_0 m_1$ and \eqref{eq_bddest1} with $\Xi=0$. 

The desired path $\zeta_1$ is now essentially obtained joining the three paths $\lambda \in (-\infty, \lambda_0] \mapsto (\lambda, T_0 \mu^{-\frac{1}{p+1}} \omega_{\mu})$, $s \in [0,1] \mapsto (\lambda_0, \sigma_1(s))$ and $\lambda \in [\lambda_0, +\infty) \mapsto (\lambda, T_1 \omega_{\mu})$, and then regularizing. We give details for the reader's convenience.

To proceed, for $\delta \in (0,1)$, consider the %a 
compact set 
\[
K_\delta \coloneqq \sigma_1([0,1]) \cup \big\{ (1-\theta) T_1 \omega_{\mu_0} + \theta T_1 \omega_\mu \mid 0 \leq \theta \leq 1, \ \lambda_0 \leq \lambda \leq \lambda_0+3\delta \big\}
\subset E. 
\]
It is not hard to find $\delta_0 \in (0,1)$ such that
\begin{equation}\label{eq_est_ab3}
	\sup \big\{ I_{G_1}(\lambda, u) \mid (\lambda, u) \in [ \lambda_0 - 3\delta_0 , \lambda_0 + 3\delta_0 ] \times N_{\delta_0} ( K_{\delta_0} ) \big\}
	< -2A_1 m_1 - 2 - \mu_0 m_1,
\end{equation}
where $N_\delta (A) \coloneqq \{ u \in E \mid \distE (u,A) < \delta \}$. Set 
\[
L_0 \coloneqq \sup_{u \in N_{\delta_0} (K_{\delta_0})} \norm{u}_2^2. 
\]
From $\abs{\Xi(s)} \leq \norm{\Xi}_X s^2 /2$, it follows that for each $\rho>0$ and $\Xi \in B_\rho^X$, 
\[
\begin{aligned}
	&\sup_{ u \in N_{\delta_0}(K_{\delta_0}) } 
	\int_{\RN} \abs{\Xi( u )} \, dx \leq 
	\sup_{ u \in N_{\delta_0}(K_{\delta_0}) } \frac{\norm{\Xi}_X}{2} \norm{u}_2^2 
	\leq \frac{L_0}{2} \rho. 
\end{aligned}
\]
Since 
\[
\abs{I_{G_1+\Xi} (\lambda, u) - I_{G_1} (\lambda, u)} \leq \int_{\RN} \abs{\Xi(u)} \,dx \leq \frac{L_{0}}{2} \rho 
\]
holds for each $\Xi \in B^X_{\rho}$ and $(\lambda,u) \in [\lambda_0 - 3 \delta_0 , \lambda_0 + 3\delta_0] \times N_{\delta_0} (K_{\delta_0})$, 
by \eqref{eq_est_ab3}, there exists sufficiently small $\rho_0 \in (0,1/3)% \blue{\min\{1/3, 2/L_0\}?}
$ such that the following uniform estimate holds
\begin{equation}\label{eq_est_ab4}
	\sup \Big\{ I_{G_1+\Xi}(\lambda, u) \mid (\lambda, u) \in [ \lambda_0 - 3\delta_0 , \lambda_0 + 3\delta_0 ] \times N_{\delta_0} ( K_{\delta_0} ), \, \Xi \in B^X_{\rho_0} \Big\}
	< -2A_1 m_1 - 2 %\blue{1?} 
- \mu_0 m_1. 
\end{equation}
Define a map $\tilde{\zeta}_1 \in C(\R, E)$ by 
\[
\tilde{\zeta}_1 (\lambda) \coloneqq \begin{dcases}
	T_0 \mu^{-1/(p+1)} \omega_\mu & \text{if} \ \lambda \in (-\infty,\lambda_0],\\
	\sigma_1 \left( (\lambda - \lambda_0) / \delta_0 \right) & \text{if} \ \lambda \in (\lambda_0,\lambda_0 + \delta_0],\\
	\left( 1 - \frac{\lambda - (\lambda_0 + \delta_0)}{\delta_0} \right) T_1 \omega_{\mu_0} + \frac{\lambda - (\lambda_0 + \delta_0)}{\delta_0} T_1 \omega_\mu 
	&\text{if} \ \lambda \in (\lambda_0 + \delta_0, \lambda_0 + 2 \delta_0],\\
	T_1 \omega_\mu & \text{if} \ \lambda \in (\lambda_0+2\delta_0, \infty).
\end{dcases}
\]
By cutting off and mollifying $\tilde{\zeta}_1$ on $[\lambda_0 - 3\delta_0 , \lambda_0 + 3\delta_0 ]$, 
there exists a sequence $(\zeta_{1,n})_{n \in \N} \subset C^2(\R, E)$ such that 
\[
\zeta_{1,n} (\lambda) = \tilde{\zeta}_1 (\lambda) \quad \text{for $|\lambda - \lambda_0| \geq 3\delta_0$}, \quad 
\max_{\lambda_0 - 3\delta_0 \leq \lambda \leq \lambda_0 + 3\delta_0 } \norm{ \zeta_{1,n} (\lambda) - \tilde{\zeta}_1(\lambda) }_E \to 0.
\]
Hence, for a sufficiently large $n$, \eqref{eq_est_ab1}, \eqref{eq_est_ab2} and \eqref{eq_est_ab4} yield 
\begin{equation*}
	I_{G_1+\Xi} (\lambda, \zeta_{1,n}(\lambda) ) < -2A_1m_1 - 1 - \mu m_1 \quad \text{for each $\lambda \in \R$ and $\Xi \in B_{\rho_0}^X$}. 
\end{equation*}
Thus (ii) holds with $\zeta_1 \coloneqq \zeta_{1,n}$ and assertion (i) can be checked easily. % with this $\zeta_1$. 

Finally, for assertion (iii), notice that $g_1 + \xi$ satisfies \ref{(g0')} and \ref{(g1*)} with $\alpha = 0$. 
By the choice of $T_0$ and $T_1$ in the above, noting $\zeta_1 %\zeta_{1,n} 
(\lambda) = T_0 \mu^{-1/(p+1)} \omega_\mu$ for $\lambda \ll -1  %1
$, 
$\zeta_1 %\zeta_{1,n} 
(\lambda) = T_1 \omega_\mu$ for $\lambda \gg 1$ and following the argument in \cite[Proof of Lemma 3.3 (ii) and (iv)]{CGIT24}, 
we obtain the desired convergence. 
\end{proof}

Let $\zeta_1 \in C^2(\R, E)$ and $\rho_0 \in (0,1/3)$ be as in \cref{Lem_zeta1rho1}, 
and set 
\begin{equation}\label{Def-gamma1}
	\gamma_1(t,\lambda) \coloneqq ( \lambda , t \zeta_1(\lambda) ).
\end{equation}
Then for each $\Xi \in B^X_{\rho_0}$, the minimax value for $I_{G_1+\Xi}$ and the uniform 2-dimensional pass class can be defined as follows (recall \eqref{eq_def_calC(L)}): 
\[
\begin{aligned}
	\overline{b}_{G_1+\Xi} &\coloneqq \inf_{\gamma \in \oGamma} \sup_{ (t,\lambda) \in [0,1] \times \R} I_{G_1+\Xi} (\gamma (t,\lambda)), 
	\\
	\oGamma &\coloneqq \big\{ \gamma \in C( [0,1] \times \R , \R \times E ) \mid \gamma = \gamma_1 \ \text{on $\calC%C
(L_\gamma)$ for some $L_\gamma > 2 %0
$}\big\}.
\end{aligned}
\]
Our next aim is to prove that there exists a $\rho_1 \in (0,\rho_0]$ such that 
for each $\Xi \in B^X_{\rho_1}$, 
$\ob_{G_1+\Xi} >0$ corresponds to a positive solution of \eqref{eq_main} with $m=m_1$ and $g=g_1+\Xi'$. 
For this purpose, we first prove the following result. 

\begin{Lemma}
	Let $\gamma_1$ be defined as in \eqref{Def-gamma1}. Then
\begin{equation}\label{eq_bdd_ob_G1Xi}
	\sup_{ \Xi \in B_{\rho_0}^X } \ob_{G_1 + \Xi} \leq \sup_{\Xi \in B_{\rho_0}^X} \sup_{ (t,\lambda) \in [0,1] \times \R} I_{G_1+\Xi} ( \gamma_1(t,\lambda) ) 
	\eqqcolon \overline{M} < \infty.
\end{equation}
The number $\overline{M}$ will take the place of the value $\beta$ in \cref{Prop:non-ex}, see the proof of \cref{Thm:Exw/og2}.
\end{Lemma}

\begin{proof}
From the definition of $\gamma_1$ and the property of $\zeta_1$, 
there exists $C_1>0$ such that for every $\lambda \leq \lambda_0 - 4\delta_0$, $s \in \R$ and $\Xi \in B^X_{\rho_0}$,  
\[
\gamma_1(t,\lambda) = \left( \lambda, t T_0 \mu^{-\frac{1}{p+1}} \omega_\mu \right), \quad 
\left| G_1(s)+\Xi (s) \right| \leq C_1 |s|^{p+1}. 
\]
Hence, \eqref{eq_scaling_omega} gives
\[
\sup_{\Xi \in B^X_{\rho_0}} \sup_{ (t,\lambda) \in [0,1] \times (-\infty,\lambda_0-4\delta_0] } I_{G_1+\Xi} (\gamma_1(t,\lambda)) < \infty. 
\]
On the other hand, when $\lambda \geq \lambda_0 + 4 \delta_0$, by writing $H_1(s) \coloneq G_1(s) - |s|^{p+1}/ (p+1)$, 
in view of \ref{(g1*)} with $\alpha = 0$ and the definition of $X$ (see \eqref{def:X} and \eqref{def:normX}), 
there exists $C_2>0$ such that for every $s \in \R$ and $\Xi \in B_{\rho_0}^X$, 
\[
\gamma_1(t,\lambda) = \left( \lambda, t T_1 \omega_\mu \right), 
\quad 
G_1(s)+\Xi (s) = \frac{|s|^{p+1}}{p+1} + H_1(s) + \Xi(s) \geq \frac{|s|^{p+1}}{p+1} - C_2 |s|^2.
\]
From \eqref{eq_scaling_omega} it follows that 
\[
\sup_{\Xi \in B^X_{\rho_0}} \sup_{ (t,\lambda) \in [0,1] \times [\lambda_0 + 4 \delta_0,\infty) } I_{G_1+\Xi} (\gamma_1(t,\lambda)) < \infty. 
\]
Finally, for $\lambda \in [\lambda_0-4\delta_0,\lambda_0 + 4\delta_0]$, 
since there exists $C_3>0$ such that $|G_1(s) + \Xi(s)| \leq C_3 |s|^{p+1}$ holds for all $s \in \R$ and $\Xi \in B^X_{\rho_0}$, 
it is easily seen from the compactness of $\gamma_1( [0,1] \times [\lambda_0 - 4\delta_0 , \lambda_0 + 4\delta_0] )$ in $\R \times E$  that 
\[
\sup_{\Xi \in B^X_{\rho_0}} \sup_{ (t,\lambda) \in [0,1] \times [\lambda_0 - 4\delta_0 , \lambda_0 + 4 \delta_0 ] } I_{G_1+\Xi} (\gamma_1(t,\lambda)) < \infty. 
\]
Thus, the fact $\gamma_1 \in \oGamma$ and the above three inequalities yield \eqref{eq_bdd_ob_G1Xi}. 
\end{proof}

Next, for $\Xi \in B_{\rho_0}^X$ and $\lambda \in \R$, recall $\mu = e^\lambda$ and set
	\[
b_{G_1+\Xi} (\lambda) \coloneqq \inf_{ \gamma \in \Lambda_{\mu,1} } \max_{0 \leq t \leq 1} 
I_{G_1+\Xi} (\lambda, \gamma(t)), \quad 
\Lambda_{\mu,1} \coloneqq \big\{\gamma \in C([0,1] , E) \mid \gamma(0) = 0, \ \gamma(1) = \zeta_1(\lambda) \big\}.
\]
Remark that $b_{G_1+\Xi} (\lambda) + \mu m_1$ is the least energy of $\Psi_{\mu, G_1+\Xi}$. 
Furthermore, \cite[Proposition 4.2 (ii)]{CGIT24} implies 
\[
b_{G_1+\Xi} (\lambda) \leq \ob_{G_1+\Xi} \quad \text{for all $\Xi \in B^X_{\rho_0}$ and $\lambda \in \R$}. 
\]
In particular, \eqref{eq_extra_blambda} yields 
\[
0 < \tilde{b}_{G_1} = \sup_{\lambda \in \R} b_{G_1} (\lambda) \leq \ob_{G_1}.
\]
Remark that by \cref{prop_b_neg_post}, if $G(s) \leq \frac{1}{p+1} \abs{s}^{p+1}$ holds everywhere, strict in at least one point, then $\tilde{b}_G>0$ (and thus $\ob_G>0$).

To proceed, we prove the following result on $b_{G_1+\Xi} (\lambda)$ and $\tilde{b}%\tilde{g}
_{G_1+\Xi} \coloneqq \sup_{\lambda \in \R} b_{G_1+\Xi}(\lambda) $. %: 

\begin{Lemma}\label{Lem:conti}
The following facts hold.
	\begin{enumerate}[label={\rm (\roman*)}]
		\item 
		The map $\R\times B_{\rho_0}^X \ni (\lambda, \Xi) \mapsto b_{G_1+\Xi} (\lambda)$ is continuous. 
		\item 
		There exists $\rho_0' \in (0,\rho_0]$ such that 
		the map $B^X_{\rho_0'} \ni \Xi \mapsto \tilde{b}_{G_1+\Xi}$ is lower semicontinuous and 
		$\tilde{b}_{G_1+\Xi} > 0$ holds for every $\Xi \in B^X_{\rho_0'}$. 
	\end{enumerate}
\end{Lemma}

\begin{proof}
	(i) 
	Since $\R \times B_{\rho_0}^X \times E \ni (\lambda,\Xi,u) \mapsto I_{G_1+\Xi} (\lambda, u) \in \R$ is continuous, 
	the minimax characterization yields the upper semicontinuity of $b_{G_1+\Xi} (\lambda)$. 
	On the other hand, let $(\lambda_n, \Xi_n) \in \R \times B^X_{\rho_0}$ satisfy $(\lambda_n , \Xi_n) \to (\lambda, \Xi)$ in $\R \times X$. 
	Choose a least energy solution $u_n \in E \setminus \{0\}$ for $\Psi_{\mu_n,G_1+\Xi_n}$, namely, $\Psi_{\mu_n,G_1+\Xi_n}' (u_n) = 0$ and (recall \eqref{eq_amu_blambda})
	$0<\Psi_{\mu_n, G_1+\Xi_n} (u_n) = a_{G_1+\Xi_n} (\mu_n) = b_{G_1+\Xi_n} (\lambda_n) + \mu_n m_1 \leq \overline{M} + \mu_n m_1$ due to \eqref{eq_bdd_ob_G1Xi}. 
	By the Pohozaev identity and the bound on the energy, $\norm{\nabla u_n}_2$ is bounded. If $N\geq 3$, exploiting again the Pohozaev identity and 
	that $\abs{G_1(t)} + \abs{\Xi_n(t)} \leq \delta t^2 + C_{\delta} t^{2^*}$ for $\delta$ small and some $C_{\delta}>0$, 
	we obtain also the boundedness of $\norm{u_n}_2$; in the case $N=2$, we may argue as in the proof of \cref{Prop:non-ex} (see also \cite{HIT10}). 

	Thus, up to a subsequence, $u_n \rightharpoonup u$ in $H^1(\R^N)$, which implies, by compact embeddings,
	$\Psi_{\mu, G_1+\Xi}'(u)=0$ and $\Psi_{\mu, G_1+\Xi}(u)\leq 
	\liminf_{n \to \infty} \Psi_{\mu_n, G_1+\Xi_n}(u_n) = \liminf_{n \to \infty} ( b_{G_1+\Xi_n} (\lambda_n) + \mu_n m_1) $.
	Notice also that $u_n,u \in C^2(\RN)$ and %(up to a subsequence) 
$u_n \to u$ in $C_{loc}(\RN)$ due to elliptic regularity and radial symmetry (see \cite[Lemma 1]{BeLi83I}). 
	To prove $u \not \equiv 0$, let $x_n$ be a maximum point of $u_n$ (remark that without loss of generality, we may suppose $u_n > 0$). 
	Since 
	\[
	e^{\lambda_n} u_n(x_n) \leq -\Delta u_n(x_n) + e^{\lambda_n} u_n(x_n) = g_1(u_n(x_n)) + \Xi_n'(u_n(x_n))
	\]
	and $\lim_{s \to 0}(g_1(s) + \Xi_n'(s) ) / s \to 0$ uniformly with respect to $n$ due to $\| \Xi_n - \Xi \|_X \to 0$, 
	the fact $e^{\lambda_n} \to e^{\lambda}$ with $u_n \not \equiv 0$ yields $\inf_{n \in \N} u_n(x_n) > 0$. 
	By Radial lemma and the boundedness of $(u_n)_{n \in \N}$, the sequence $(x_n)_{n \in \N}$ is bounded. 
	Recalling $u_n \to u$ in $C_{loc} (\RN)$ and the strong maximum principle, we have $u > 0$ and $b_{G_1+\Xi} (\lambda) \leq \Psi_{\mu,G_1+\Xi} (u)+\mu m_1$. 
%	Furthermore, by $u_n \not \equiv 0$, $\Psi_{\mu_n, G_1+\Xi_n}'(u_n) = 0$ and $\Xi_n \to \Xi$ in $X$, 
%	it is easily seen that \blue{M: \underline{to be justified}} $u \not \equiv 0$ and $b_{G_1+\Xi} (\lambda) \leq \Psi_{\mu,G_1+\Xi} (u)\red{+\mu m_1}$. 
	This leads to $b_{G_1+\Xi}(\lambda) \leq \liminf_{n \to \infty} b_{G_1+\Xi_n}(\lambda_n)$ and (i) holds. 
	
	(ii)
	By (i) and the definition of $\tilde{b}_{G_1+\Xi}$, the map $B_{\rho_0}^X \ni \Xi \mapsto \tilde{b}_{G_1+\Xi}$ is lower semicontinuous. 
	From $\tilde{b}_{G_1} > 0$ and the lower semicontinuity of $\tilde{b}_{G_1+\Xi}$ on $\Xi$, 
	it follows that there exists $\rho_0' \in (0,\rho_0]$ such that $\tilde{b}_{G_1+\Xi} > 0$ holds for every $\Xi \in B^X_{\rho_0'}$. 
\end{proof}

\begin{proof}[Proof of \cref{Thm:Exw/og2}]
	Let $\rho_0' >0$ be as in \cref{Lem:conti} (ii) and set $\beta \coloneqq \overline{M}$ where $\overline{M}$ appears in \eqref{eq_bdd_ob_G1Xi}. 
Observe that $g_1$ satisfies $\tilde{b}_{G_1}>0$, and \ref{(g2w)} thanks to \ref{(g2)}. 
Let $\rho_1=\rho_1(g_1,\beta, \rho_0')=\rho_1(g_1)>0$ be the smallest between $\rho_0'$ and the radius $\rho_{\beta}$ %radii 
given by \cref{Prop:non-ex}, 
thus $g_1+\Xi'$ satisfies \ref{(g2w)} and $\beta %\overline{M} 
\geq \ob_{G_1+\Xi} > 0$ 
for each $\Xi \in B^X_{\rho_1} $. 
Fix $\Xi \in B^X_{\rho_1}$ arbitrarily. 
Since $g_1 + \Xi'$ satisfies \ref{(g1*)} with $\alpha = 0$ and \ref{(g2w)}, 
by $\beta \geq \ob_{G_1 + \Xi} > 0$, we may verify that $I_{G_1+\Xi}$ enjoys the $\PSPC_{ \ob_{G_1 + \Xi} }^*$ condition 
as in \cite[Proposition 5.4]{CGIT24} (see also \cref{D:PSPC} and \cref{prop_compact_Kb}). 
Notice that \ref{(g2)} is used in \cite[Subsection 5.4]{CGIT24} and 
the argument is still valid under \ref{(g2w)} provided $0 < \ob_{G_1 + \Xi} \leq \beta$. 
Therefore, we may utilize the deformation argument in \cite[Subsection 8.2%Section 8
]{CGIT24}, and
obtain the existence of a positive solution corresponding to $\ob_{G_1 + \Xi}$  through a minimax contradiction argument, see \cite[Subsection 8.5]{CGIT24}.
This concludes the proof.
\end{proof}

\begin{proof}[Proof of \cref{Cor:Exw/og2}]
It follows by \cref{prop_b_neg_post} and \cref{Thm:Exw/og2}.
\end{proof}

\begin{proof}[Proof of \cref{Cor:ex-g}]
Let $N \geq 3$. 
We first find $g_1$ satisfying \ref{(g0')}, \ref{(g1*)} with $\alpha = 0$, \ref{(g2)} and \eqref{eq_condg_1}. 
Since \ref{(g3)} %\ref{(g2)} 
is equivalent to $F_1' \leq 0$ where $F_1 (s) \coloneqq s^{-2^*} G_1(s)$, we select $F_1 \in C^1((0,\infty))$ so that 
\[
\begin{dcases}
	F_1(s) = \frac{s^{p+1-2^*}}{p+1} \ \ \text{for $s \in (0,1) \cup (3,\infty)$}, \quad 
	F_1(s) \leq \frac{s^{p+1-2^*}}{p+1} \quad \text{for $s \in (1,3)$}, 
	\\
	F_1' \leq 0 \quad \text{in} \ (0,\infty), \quad F_1(2) < \frac{2^{p+1-2^\ast}}{p+1}, \quad F_1'(2) = 0.
\end{dcases}
\]
Notice that such an $F_1$ is easily constructed. 
Then $G_1(s) \coloneq s^{2^*} F_1(s)$ and $g_1 \coloneq G_1'$ satisfy \ref{(g1*)} with $\alpha = 0$, 
\ref{(g2)} and \eqref{eq_condg_1}. % with $s_0=2$. 
By extending $G_1$ as an even function, \ref{(g0')} also holds.
By \cref{Cor:Exw/og2} and the proof of \cref{Thm:Exw/og2}, consider $\rho_1>0$ such that, for any $\Xi \in B^X_{\rho_1} %X
$, 
$g_1 + \Xi'$ satisfies \ref{(g2w)} for $\beta = \overline{M} %\ob_{G_1}+1 % \ob_{G_1-\Xi_1} 
>0$ in \eqref{eq_bdd_ob_G1Xi} 
and $\ob_{G_1 + \Xi} > 0$ corresponds to a positive solution of \eqref{eq_main} with $g=g_1+\Xi'$. 
In particular, for an even, non-negative $\varphi \in C^\infty(\R)$ with $\supp \varphi_{|(0,+\infty)} \subset 
(1,3)$ and $\varphi'(2) > %< 
0$, 
consider $\Upsilon (s) \coloneqq \abs{s}^{2^*} \varphi(s)$ and $G \coloneqq G_1 + \eps \Upsilon $. For a sufficiently small $\eps>0$, 
we have $\eps \Upsilon \in X$ 
 and $ \norm{\eps \Upsilon }_X < \rho_1$.
Therefore, \eqref{eq_main} with $g\coloneqq g_1 + \eps \Upsilon'$ admits a positive solution corresponding to $\ob_{G_1 + \eps \Upsilon} > 0$ 
and $g_1 + \eps\Upsilon'$ satisfies \ref{(g2w)}. On the other hand, by 
\[
\frac{d}{ds} \left( \frac{G_1(s) + \eps \Upsilon}{s^{2^*}} \right)\Big|_{s=2} 
= F_1'(2) + \eps \varphi'(2) = \eps \varphi'(2) > %< 
0,
\]
$g_1 + \eps \Upsilon'$ does not satisfy \ref{(g3)}. Thus \cref{Cor:ex-g} holds. 

Let us finally deal with the case $N=2$. In this case, 
\ref{(g3)} becomes $G(s) \geq 0$ for all $s \geq 0$. Consider any $G_1 \in C^1(\R)$ with the properties: 
	\[
	G_1(s) = \frac{s^{p+1}}{p+1} \quad \text{for $s \in (0,1) \cup (3,\infty)$}, \quad 
G_1(2) = 0, \quad
	G_1(s) \leq \frac{s^{p+1}}{p+1} \quad \text{for $s \in (0,\infty)$}.	
	\]
	Then $g_1 = G_1'$ satisfies the assumption in \cref{Cor:Exw/og2};
	we know there exists $\rho_1>0$ such that for any $\Xi \in% X_1 \cap
	 B^X_{\rho}$, %with $|\Xi| \leq \Xi_1$, 
	\eqref{eq_main} has a positive solution with $g=g_1+\Xi'$. 
	Thus, if $\varphi \in C^\infty$ satisfies $\supp \varphi \subset %\Set{ s \in (0,\infty) | \Xi(s) > 0 } \subset 
	(1,3)$ and $\varphi(2) < 0$, 
	then $g_1 + \eps \varphi$ for sufficient small $\eps>0$ is the desired example. % for $N=2$. 
\end{proof}

%%%%%%%%%%%%%%%%%%%%%%%%%%%%%%%%%%%%%%%%%%%%%%%%%%%%%%%%%%%%%%%%%%
\section*{Acknowledgments}
The authors would like to warmly thank the anonymous referee for many interesting comments which improved the paper, 
	and in particular for suggesting to us the weakening of condition \ref{(rho+1)}.

\section*{Fundings}
S.$\,$C. is partially supported by INdAM-GNAMPA and by PRIN PNRR P2022YFAJH \lq\lq Linear and 
Nonlinear PDEs: New directions and applications''. 
S.$\,$C. thanks acknowledge financial support from PNRR MUR \sloppy
project PE0000023 NQSTI 
- National Quantum Science and Technology Institute (CUP H93C22000670006). 
M.$\,$G. is supported by INdAM-GNAMPA Project 
\lq \lq Metodi variazionali per problemi dipendenti da operatori frazionari isotropi e anisotropi'',
 codice CUP \#E5324001950001\#.
N.$\,$I. is partially supported by JSPS KAKENHI Grant Numbers JP 19H01797, 19K03590 and 24K06802. 
K.$\,$T. is partially supported by JSPS KAKENHI Grant Numbers JP18KK0073, JP19H00644 and 
JP22K03380.

%%%%%%%%%%%%%%%%%%%%%%%%%%%%%%%%%%%%%%%%%%%%%%%%%%%%%%%%%%%%%%%%%

\end{document}